\RequirePackage[normalem]{ulem} 
\RequirePackage{color}\definecolor{RED}{rgb}{1,0,0}\definecolor{BLUE}{rgb}{0,0,1} 

\documentclass[11pt, reqno, dvipsnames]{amsart}
\usepackage[utf8]{inputenc}
\usepackage[english]{babel}
\usepackage{color}
\usepackage{amsmath,amsthm,amssymb,amsfonts}
\usepackage{mathrsfs}
\usepackage{tikz-cd}
\usepackage[linktocpage=true]{hyperref}
\usepackage{enumerate}
\usepackage{geometry}

\hypersetup{
	colorlinks = true,
	linkcolor = {ForestGreen},
	citecolor = {RoyalBlue},
	urlcolor = {Black}
}

\newcommand{\Z}{{\mathbb Z}}
\newcommand{\Q}{{\mathbb Q}}

\newcommand{\N}{{\mathbb N}}

\newcommand{\M}{{\mathcal M}}

\newcommand{\f}{{\mathcal F}}

\newcommand{\Spec}{\operatorname{Spec}}
\newcommand{\Sp}{\operatorname{Sp}}
\newcommand{\Spa}{\operatorname{Spa}}
\newcommand{\Spf}{\operatorname{Spf}}
\newcommand{\rg}{\operatorname{R\Gamma}}

\newcommand{\colim}{\operatorname{colim}}

\newcommand{\cond}{\operatorname{cond}}
\newcommand{\Mod}{\operatorname{Mod}}

\newcommand{\et}{\operatorname{\acute{e}t}}
\newcommand{\ket}{\operatorname{k\acute{e}t}}
\newcommand{\proet}{\operatorname{pro\acute{e}t}}
\newcommand{\proket}{\operatorname{prok\acute{e}t}}

\newcommand{\fil}{\operatorname{Fil}}
\newcommand{\gr}{\operatorname{gr}}
\newcommand{\logdr}{\operatorname{logdR}}

\numberwithin{equation}{section}
\newtheorem{theorem}{Theorem}
\numberwithin{theorem}{section}

\newtheorem{thm}[theorem]{Theorem}
\newtheorem{lem}[theorem]{Lemma}
\newtheorem{cor}[theorem]{Corollary}
\newtheorem{prop}[theorem]{Proposition}

\theoremstyle{definition}
\newtheorem{defn}[theorem]{Definition}
\newtheorem{rem}[theorem]{Remark}

\AtBeginDocument{%
	\def\MR#1{}
}

\numberwithin{equation}{section}
\begin{document}
	
	\title{On the pro-Kummer \'etale cohomology of $\mathbb B_{\operatorname{dR}}$}
	
	\author{Xinyu Shao}
	
	\address{Institut de Mathématiques de Jussieu-Paris Rive Gauche, Sorbonne Universit\'e, 4 place Jussieu, 75005 Paris, France}
	
	\address{present address: Institut für Algebra, Zahlentheorie und Diskrete Mathematik, Leibniz Universität Hannover, Welfengarten 1, 30167 Hannover, Germany}
	
	\email{xinyu.shao@math.uni-hannover.de}
	
	\begin{abstract}
		We investigate $p$-adic cohomology theories of log rigid analytic varieties over a $p$-adic field. For a log rigid analytic variety $X$ defined over a discretely valued field, we compute the Kummer pro-\'etale cohomology of $\mathbb{B}_{\mathrm{dR}}^+$ and $\mathbb{B}_{\mathrm{dR}}$. When $X$ is defined over $\mathbb{C}_p$, we introduce a logarithmic ${B}_{\mathrm{dR}}^+$-cohomology theory, which serves as a deformation of log de Rham cohomology. Additionally, we establish the log de Rham-\'etale comparison in this setting and prove the degeneration of both the Hodge-Tate and Hodge-log de Rham spectral sequences when $X$ is proper and log smooth.  
	\end{abstract}
	
	\maketitle
	\tableofcontents	
	
	\section{Introduction}

	Let $\mathscr{O}_K$ be a complete discrete valuation ring with fraction field $K$ of characteristic 0 and with perfect residue field $k$ of characteristic $p$. Let $W(k)$ be the ring of Witt vectors of $k$ with fraction field $F$ (therefore $W({k})=\mathscr{O}_F$). Let $\bar{K}$ be an algebraic closure of $K$ and $C$ be its $p$-adic completion, and let $\mathscr{O}_{\bar{K}}$ be the integer closure of $\mathscr{O}_K$ in $\bar{K}$. Let $\phi$ be the absolute Frobenius on $W(\bar{k})$. Set $\mathscr{G}=\operatorname{Gal}(\bar{K}/K)$.	
	
	We will denote by $\mathcal{O}_K$ and $\mathcal{O}_K^{\times}$,  depending on the context, either the scheme $\Spec(\mathcal{O}_K)$ or the formal scheme $\Spf(\mathcal{O}_K)$ with the trivial log structure or the canonical (i.e., associated to the closed point) log structure, respectively.
	
	\subsection{Background}
	
	Before presenting the main theorem, we review some recent advances in $p$-adic Hodge theory.
	
	The aim of $p$-adic Hodge theory is to establish a Hodge-type decomposition over $p$-adic fields, analogous to the classical Hodge theory developed by Deligne and Hodge. For smooth algebraic varieties, such a decomposition was first achieved in \cite{faltings1988padichodge}. In the case of proper smooth rigid analytic varieties, Scholze extended this framework using perfectoid geometry, leading to the following results in \cite{scholze2013p}.
	
	\begin{thm} \cite[Corollary 1.8]{scholze2013p}
		For any proper smooth rigid-analytic variety $X$ defined over $K$, the Hodge--de Rham spectral sequence $$E^{ij}_1:=H^j(X,\Omega^i_{X})\Rightarrow H^{i+j}_{\operatorname{dR}}(X)$$ degenerates at $E_1$, and there is a Hodge-Tate decomposition $$H^i_{\et}(X_C,\Q_p)\otimes_K C \simeq \bigoplus_{j=0}^i H^{i-j}(X,\Omega_X^j)\otimes_K C(-j).$$ Moreover, the $p$-adic \'etale cohomology $H^i_{\et}(X_C,\Q_p)$ is de Rham with associated filtered $K$-vector space $H^i_{\operatorname{dR}}(X).$
	\end{thm}
	
	Following Scholze’s celebrated work on proper smooth rigid analytic varieties, numerous generalizations have emerged. The work by Hansheng Diao, Kai-Wen Lan, Ruochuan Liu, and Xinwen Zhu in \cite{dllz2023logrh} (see also the work of Shizhang Li and Xuanyu Pan in \cite{lipan2019log}) extended the above comparison theorem to almost proper rigid analytic varieties $X$, meaning that $ X $ can be written as $ \overline{X} - Z $ where $ \overline{X}$ is a proper smooth variety and $ Z $ is a Zariski closed subset in $ \overline{X}$. In \cite{dllz2023logrh}, they further constructed the Simpson and Riemann-Hilbert correspondences for such varieties. Their approach involves developing the theory of log adic spaces, which provides a natural log structure on $ \overline{X} $ induced by the open immersion $ \overline{X} - Z \hookrightarrow \overline{X} $. By employing resolution of singularities, $ Z $ can be assumed to be a strictly normal crossing divisor on $ \overline{X} $, allowing one to reduce the problem on $ X $ to one on $ \overline{X} $, where methods akin to those in \cite{scholze2013p} apply.
	
	Another direction of generalization is to study general rigid analytic varieties without requiring $ X $ to be proper or smooth. This was studied in \cite{bosco2023padicproetalecohomologydrinfeld} and \cite{bosco2023rational}. In \cite[Theorem 1.1.8]{bosco2023padicproetalecohomologydrinfeld}, Guido Bosco computed the pro-\'etale cohomology of the period sheaf $ \mathbb{B_{\operatorname{dR}}} $ for smooth rigid analytic varieties over a discretely valued field $ K $. Within the framework of condensed mathematics, he computed the pro-\'etale cohomology of $ \mathbb{B_{\operatorname{dR}}} $ via the Poincaré lemma, which, for proper smooth rigid analytic varieties over $ K $, reproduces results from \cite[Theorem 7.11]{scholze2013p}. In \cite{bosco2023rational}, using the \'eh-topology, Guido Bosco showed it is also possible to work with non-smooth rigid analytic varieties over an algebraically closed complete field $C$. 
	
	\subsection{The pro-Kummer \'etale cohomology of $\mathbb B_{\operatorname{dR}}$}
	
	In this article, we compute the pro-Kummer \'etale cohomology of $\mathbb B_{\operatorname{dR}}$ for log-smooth rigid analytic varieties over a discretely valued field $ K $, extending the results of \cite[Theorem 1.1.8]{bosco2023padicproetalecohomologydrinfeld}. 
	
	In algebraic geometry, for a given smooth algebraic variety $ X $, we typically consider a compactification $ \overline{X}$ such that $ \overline{X} - X $ is a strict normal crossing divisor. This setup allows $ \overline{X} $ to carry a natural log structure that is log-smooth, thereby reducing many problems to the proper case where cohomology theory behaves more predictably. However, in rigid analytic geometry, compactifying a rigid analytic variety is generally not feasible (while it can be done in the category of adic spaces, the resulting space no longer remains a rigid analytic variety, introducing new complications). This is why, in \cite{dllz2023logrh}, the authors restrict themselves to considering almost proper rigid analytic varieties, which can be associated with a smooth proper rigid analytic variety endowed with a well-behaved log structure after resolution of singularities.
	
	In order to extend the known comparison theorems for rigid analytic varieties, it is natural to introduce log structures on general (not necessarily proper) rigid analytic varieties. For pro-\'etale cohomology, recent works by Pierre Colmez, Gabriel Dospinescu, and Wiesława Nizioł \cite{CNsytntomic}, \cite{CDNstein}, \cite{CNderham}, \cite{CN4.3} provide a detailed description of the pro-\'etale cohomology for smooth rigid analytic varieties, thanks to the comparison with syntomic cohomology. However, for non-quasi-compact rigid analytic varieties, computing \'etale cohomology remains challenging. By equipping a quasi-compact rigid analytic variety with a log structure derived from a strict normal crossing divisor, we can apply results on pro-\'etale cohomology to compute the \'etale cohomology of the trivial locus, which is itself not quasi-compact. Further exploration in log geometry, therefore, holds potential to clarify significant aspects of the \'etale cohomology of rigid analytic varieties.
	
	The main theorem of this part is as follows (see Theorem \ref{bdrcoh} for a more general statement).
	
	\begin{thm}
		Let $X$ be a log smooth rigid analytic variety defined over $K$ with fine saturated log structure. Define $$\rg_{\underline{\operatorname{logdR}}}(X_{B_{\operatorname{dR}}}):=\rg(X,\underline{\Omega_X^{\bullet,\log}}\otimes_K^{\blacksquare}\underline{B_{\operatorname{dR}}}).$$
		
		(i) We have a $\mathscr{G}_K$-equivariant, compatible with filtrations, natural isomorphisms in $D(\Mod_K^{\cond})$: $$\rg_{\underline{\proket}}(X_C,\mathbb B_{\operatorname{dR}})\simeq \rg_{\underline{\operatorname{logdR}}}(X_{B_{\operatorname{dR}}}).$$
		
		(ii) Assume $X$ is connected and paracompact. Then, for each $r\in \Z$, we have a $\mathscr{G}_K$-equivariant natural quasi-isomorphism in $D(\Mod_K^{\cond})$: $$\rg_{\underline{\proket}}(X_C,\fil^r\mathbb B_{\operatorname{dR}})\simeq \fil^r(\rg_{\underline{\operatorname{logdR}}}(X)\otimes_K^{L_\blacksquare}\underline{B_{\operatorname{dR}}}),$$ where $\rg_{\underline{\operatorname{logdR}}}(X)$ is the condensed log de Rham cohomology complex in $D(\Mod_K^{\cond})$ (see Definition \ref{deflogdr}).
	\end{thm}
	
	This result generalizes \cite[Theorem 1.1.8]{bosco2023padicproetalecohomologydrinfeld}. By comparing these two results, we can relate the pro-Kummer \'etale cohomology and the pro-\'etale cohomology of $\mathbb B_{\operatorname{dR}}$ when the log structure of $X$ comes from a strictly normal crossing divisor.
	
	\begin{cor}
		Let $X$ be a smooth rigid analytic variety over $K$, $D \subset X$ be a strictly normal crossing divisor, and $U=X-D$. Endow $X$ with the log structure coming from $D$. Then we have a natural (non-filtered!) quasi-isomorphism in $D(\Mod_K^{\cond})$: $$\rg_{\underline{\proet}}(X_C,\mathbb B_{\operatorname{dR}}) \simeq \rg_{\underline{\proket}}(U_C,\mathbb B_{\operatorname{dR}}).$$
	\end{cor}
	
	It is unclear whether this corollary can be proved directly, as computations for Kummer (pro-)\'etale cohomology typically require case-by-case treatment. For instance, if $X$ is quasi-compact, $D \subset X$ is a strictly normal crossing divisor, and $U=X-D$, then we have $$\rg_{\underline{\proket}}(X_C,\Q_p) \simeq \rg_{\underline{\ket}}(X_C,\Q_p) \simeq \rg_{\underline{\et}}(U_C,\Q_p),$$ which follows from \cite[Corollary 6.3.4]{dllz2023logadic}. Note that $\rg_{\underline{\et}}(U_C,\Q_p) \neq \rg_{\underline{\proet}}(U_C,\Q_p)$ in general. For example, when $U=\mathbb A^1,$ $H^1_{\et}(A^1_C,\Q_p)=0$ but $H^1_{\proet}(A^1_C,\Q_p)$ is very large, as indicated by the main theorem of \cite{CDNstein}.
	
	\subsection{Geometric $ B_{\operatorname{dR}}^+$-cohomology and comparison theorem}
	
	Now let $X$ be a log smooth rigid analytic variety over $C$. Since the algebra ${B}_{\operatorname{dR}}^+$ lacks a natural $C$-structure, establishing a meaningful comparison theorem between the log de Rham cohomology and the Kummer \'etale cohomology necessitates the introduction of a new cohomology theory, $\rg_{B^+_{\operatorname{dR}}}(X)$. This theory serves as a deformation of log de Rham cohomology via the map $ B^+_{\operatorname{dR}} \to C $.  
	
	For rigid analytic varieties without log structures, various approaches have been proposed to construct geometric $ B_{\operatorname{dR}}^+$-cohomology theories. For example:
	
	\begin{itemize}
		\item In \cite{bms1}, Bhargav Bhatt, Matthew Morrow, and Peter Scholze introduce $\rg_{{\operatorname{cris}}}(X/B^+_{\operatorname{dR}})$ using the concept of ``very small objects."
		
		\item In \cite{guo2021crystalline}, Haoyang Guo constructs $\rg(X/\Sigma_{\mathrm{inf}}, \mathcal{O}_{X/\Sigma_{\mathrm{inf}}})$ by studying the infinitesimal site over $B^+_{\operatorname{dR}}$.
		
		\item In \cite{CN4.3}, Pierre Colmez and Wiesława Nizioł define $\rg_{{\operatorname{dR}}}(X/B^+_{\operatorname{dR}})$ via local semistable formal models of $ X$.
	\end{itemize}
	
	Notably, all these definitions are shown to be equivalent by \cite[Proposition 3.27]{CN4.3}. In \cite{bosco2023rational}, using the $ L\eta_f $-functor, Guido Bosco introduces $\rg_{B^+_{\operatorname{dR}}}(X)$ via period sheaves, and proves that it agrees with the geometric $ B_{\operatorname{dR}}^+$-cohomology.
	
	In this article, we adopt Bosco’s approach to define $\rg_{B_{\operatorname{dR}}^+}(X)$, as we believe it provides the most direct framework for defining logarithmic ${B}_{\operatorname{dR}}^+$-cohomology with minimal prerequisites. After establishing such a theory, we will prove the following comparison theorem.

	\begin{theorem} \label{drcomparison}
		Let $X$ be a proper log smooth rigid analytic variety over $C$. Then there are cohomology groups $H^i_{B_{\operatorname{dR}}^+}(X)$ which come with a canonical filtered isomorphism $$H^i_{\ket}(X,\Q_p) \otimes_{\Q_p}B_{\operatorname{dR}} \simeq H^i_{B_{\operatorname{dR}}^+}(X) \otimes_{B_{\operatorname{dR}}^+}B_{\operatorname{dR}}.$$ When $X$ comes from $X_0$ over a discretely valued field $K$, this isomorphism agrees with the comparison isomorphism (see \cite[Theorem 1.1]{dllz2023logrh}) $$H^i_{\ket}(X,\Q_p) \otimes_{\Q_p}B_{\operatorname{dR}} \simeq H^i_{\operatorname{logdR}}(X_0) \otimes_{K}B_{\operatorname{dR}},$$ under a canonical identification $$H^i_{B_{\operatorname{dR}}^+}(X)= H^i_{{\operatorname{logdR}}}(X_0)\otimes_KB_{\operatorname{dR}}^+.$$ Moreover, $H^i_{B_{\operatorname{dR}}^+}(X)$ is a finite free $B_{\operatorname{dR}}^+$-module, and we have
		
		(i) The Hodge--de Rham spectral sequence $$E^{ij}_1:=H^j(X,\Omega^{\log,i}_{X})\Rightarrow H^{i+j}_{\operatorname{logdR}}(X)$$ degenerates at $E_1$.
		
		(ii) The Hodge--Tate spectral sequence
		$$E^{ij}_2:=H^j(X,\Omega^{\log,i}_{X})(-j)\Rightarrow H^{i+j}_{\ket}(X,\Q_p)\otimes_{\Q_p}C$$ degenerates at $E_2$.
	\end{theorem}

	These results serve as an initial step toward the study of rational $p$-adic Hodge theory for logarithmic rigid analytic varieties. In fact, we can define the syntomic cohomology for logarithmic rigid analytic varieties by adopting a similar framework. We will explore these ideas in detail in a forthcoming paper \cite{xslogsyn}.
	

\begingroup
\makeatletter
\let\@tocwrite\@gobbletwo
\makeatother
	
	\subsection*{Acknowledgments}
	
	The author would like to sincerely thank his thesis advisor, Wiesława Nizioł, for her insightful discussions, and for posing key questions that shaped this article. He thanks the anonymous referees, as well as the paper reviewers Johannes Anschütz and Federico Binda, for their numerous valuable comments, which have greatly improved the readability and rigor of the manuscript. The author also wishes to thank Ben Heuer and Weizhe Zheng for useful discussions. The article is part of the author’s Ph.D. thesis, which has been defended at Sorbonne Université, within the Institut de mathématiques de Jussieu – Paris Rive Gauche.

	This work was also funded by the Deutsche Forschungsgemeinschaft (DFG, German Research Foundation) through the Collaborative Research Centre TRR 326 ‘Geometry and Arithmetic of Uniformized Structures’, project number 444845124.
	
	\subsection*{Notation and conventions}
	
	In this article, We freely use the language of condensed mathematics as developed in \cite{cscondensed}. To avoid set-theoretic issues, we fix an uncountable strong limit cardinal $\kappa$, which means $\kappa$ is uncountable and for all $\lambda<\kappa,$ also $2^{\lambda}<\kappa.$ For sheaves valued in condensed abelian groups, we refer the reader to \cite{bosco2023padicproetalecohomologydrinfeld} for notations and properties used in this article. In particular, here are some notations that are used in this article:
	\begin{itemize}
		\item We denote by $\operatorname{CondAb}$ the category of condensed abelian groups. For a discretely valued field $K$, we denote by $\operatorname{Mod}_K^{\cond}$ and $\operatorname{Mod}_K^{\operatorname{solid}}$ the category of condensed $K$-vector spaces and the category of solid $K$-vector spaces, respectively.
		\item For a topological abelian group $T$, we can associate $T$ with a condensed abelian group $\underline{T}$ as follows: for any profinite group $S$, we define $\underline{T}(S):=\mathscr{C}(S,T)$ the set of continuous maps from $S$ to $T$. Conversely, for a condensed abelian group $A$, we denote by $A(*)$ the underlying abelian group of $A$.
		\item The above construction allows us to relate a sheaf $\f$ of topological abelian groups over a topological space $X$ with a sheaf $\underline{\f}$, by sending any open $U \subset X$ to $\underline{\f}(U):=\underline{\f(U)}.$
		\item For $\tau=\et$ or $\proet$, and an analytic adic space $X$, we can define the site $X_{\tau,\cond}$ to retaining the information captured by profinite sets, see \cite[2.3]{bosco2023padicproetalecohomologydrinfeld}.
	\end{itemize}
	
	We adopt the language of adic spaces as developed in \cite{huber2013etale}. A rigid analytic variety over $L$ is defined as a quasi-separated adic space locally of finite type over $\Spa(L,\mathcal O_L)$ for a $p$-adic field $L$. All rigid analytic spaces considered will be over $K$ or $C$. We assume all rigid analytic varieties are separated, taut, and countable at infinity. We denote by $\operatorname{Sm}_L$ the category of smooth rigid analytic varieties (or smooth dagger varieties) over $L$. Moreover, we always assume analytic adic spaces are $\kappa$-small, meaning the cardinality of their underlying topological spaces is less than $\kappa.$
	
	If $\mathcal A$ is an abelian category, unless stated otherwise, we always work with derived stable $\infty$-category $D(\mathcal{A})$.
	
	We also use the theory of log adic spaces. For definitions and properties of log adic spaces, we refer the reader to \cite{dllz2023logadic}. A (pre) log-structure on a condensed ring is simply a (pre)log-structure on the underlying ring.

\par
\endgroup
	
	
	
	
	
	

	\section{Logarithmic de Rham cohomology} 
	
	In this section, we study the de Rham cohomology for logarithmic rigid analytic varieties. We will show that most of the known results continue to hold in the condensed mathematics setting.
	
	Suppose $X$ is a log smooth rigid analytic variety over $L=K$ or $C$, i.e. an analytic fs log adic space which is locally of finite type over $\Spa(L,\mathcal O_L).$ Recall that in \cite[Construction 3.3.2]{dllz2023logadic}, one has logarithmic differential \'etale sheaf $\Omega_X^{1,\log}$ (also denoted by $\Omega_X^{\log}$), defined to be the universal object among derivations. When $X$ is log smooth, by \cite[Lemma 3.3.15]{dllz2023logadic} $\Omega_X^{1,\log}$ is locally free of finite rank. For any $i\geq 1$, let $\Omega_X^{i,\log}:=\bigwedge^i \Omega_X^{1,\log}.$ 
	
	For any affinoid \'etale morphism $U \to X,$ $\Omega_X^{i,\log}(U)$ is a finite $\mathcal{O}_U(U)$-module, therefore it has a natural structure of $K$-Banach space. This leads to the following definition.
	
	\begin{defn}\label{deflogdr}
		Suppose $X$ is a log smooth rigid analytic variety over $L$, we define the logarithmic de Rham cohomology of $X$ to be $$\rg_{\underline{\operatorname{logdR}}} (X):=\rg\left(X,\underline{\Omega_X^{\bullet,\log}}\right)$$ in $D(\Mod_K^{\cond}).$
	\end{defn}
	
	Recall that the usual de Rham cohomology for log rigid analytic varieties is defined by $$\rg_{{\operatorname{logdR}}} (X):=\rg\left(X,{\Omega_X^{\bullet,\log}}\right).$$ When $X$ is log smooth and proper, $H^i_{\operatorname{logdR}}(X)$ is finite dimensional and carries a natural topology, therefore the structure of $\rg_{\underline{\operatorname{logdR}}} (X)$ is simple.
	
	\begin{lem} \label{finitenessderham}
		Suppose $X$ is log smooth and proper, then we have $$\underline{H^i_{\operatorname{logdR}}(X)}=H^i_{\underline{\operatorname{logdR}}}(X)$$ for all $i \geq 0.$
	\end{lem}
	
	\begin{proof}
		The proof of \cite[Lemma 5.11]{bosco2023padicproetalecohomologydrinfeld} goes through.
	\end{proof}
	
	We are mainly interested in the case where the log structure of $X$ comes from a strictly normal crossing divisor $D \subset X$. Following \cite{grothendieck1966derham}, \cite{kiehl1967derham} and \cite{deligne1970equations}, we give the following definition.
	
	\begin{defn}
		Suppose $X$ is a smooth rigid analytic variety over $L$,  $D \subset X$ is a strictly normal crossing divisor, and $U=X-D.$ Let $i \geq 1.$
		
		(i) We define $\Omega_X^{i}(*D)$ to be the differential $i$-forms on $X$ with meromorphic poles along $D$. In other words, let $\mathscr{J}$ be the ideal sheaf associated to the closed immersion $D \hookrightarrow X,$ then $$\Omega_X^{i}(*D):=\varinjlim_{n}\mathcal{H}om_{\mathcal{O}_X}(\mathscr{J}^n,\Omega_X^{i}).$$
		
		(ii) We define $\Omega_X^{i}(\log D)$ to be the differential $i$-forms on $X$ with logarithmic poles along $D$. In other words, $\Omega_X^{i}(\log D)$ contains the elements $\omega$ of $\Omega_X^{i}(*D)$ such that $\omega$ and $\operatorname{d}\omega$ have a pole of order at most 1.
	\end{defn}
	
	We can also give a local description of the structure of $\Omega_X^{1}(\log D).$ Suppose that $X=\Spa(R,R^+) \to \mathbb D^d:=\Spa(L\langle x_1,...,x_d\rangle,\mathcal O_L\langle x_1,...,x_d\rangle)$ is an \'etale map, and $D \subset X$ is the divisor defined by $x_1\cdots x_d=0$. Then $\Omega_X^{i}(\log D)(X)$ is the free $R$-module generated by $\frac{dx_1}{x_1},...,\frac{dx_d}{x_d}$. For $i \geq 1$ we have $\Omega ^ p_{X}(\log D) = \wedge ^ p(\Omega ^1_{X}(\log D))$ by \cite[0FUP]{stacks-project}.
	
	On the other hand, a strictly normal crossing divisor $D \subset X$ gives $X$ a natural log structure, and we refer the reader to \cite[Example 2.3.17]{dllz2023logadic} for a description of this log structure. The next lemma compares $\Omega_X^{i,\log}$ (the universal object among derivations) and $\Omega_X^{i}(\log D).$
	
	\begin{lem}
		We have a natural isomorphism $$\Omega_X^{i,\log}\xrightarrow{\simeq}\Omega_X^{i}(\log D),$$ for each $i\geq 1,$ where the log structure $\alpha:\mathcal M \to \mathcal{O}_X$ is given by the strictly normal crossing divisor $D \subset X$.
	\end{lem}
	
	\begin{proof}
		We only need to prove the lemma for $i=1$. We recall that, by the local description of $\Omega_X^{i,\log}$ in \cite[(3.2.7)]{dllz2023logadic}, $\Omega_X^{\log}$ is the quotient of $\Omega_X \oplus (\mathcal{O}_X \otimes_{\mathbb{Z}} \mathcal M^{gp})$ by sheaf of submodules $\mathcal K$ generated by $(d\alpha(m), -\alpha(m) \otimes m)$
		
		Define $\widetilde{\phi}: \Omega_X \oplus (\mathcal{O}_X \otimes_{\mathbb{Z}} \mathcal M^{gp}) \to \Omega^1_X(\log D)$ as $\widetilde{\phi}(\omega, f \otimes m)=\omega+ f\dfrac{d(\alpha(m))}{\alpha(m)}$. Since $\widetilde{\phi}(\mathcal K)=0,$ it induces a morphism of sheaves $\phi:\Omega_X^{\log}\to\Omega_X(\log D).$ We need to show that $\phi$ is an isomorphism. We can check it \'etale locally. By \cite[Corollary 3.1.11]{dllz2023logadic}, we may assume that there exists a strictly \'etale morphism $X=\Spa(R,R^+) \to \mathbb D^d:=\Spa(L\langle x_1,...,x_d\rangle,\mathcal O_L\langle x_1,...,x_d\rangle)$, and the log structure on $\mathbb D^d$ is the log structure associated to the pre-log structure $\N^d \to L\langle x_1,...,x_d\rangle$, sending $e_i$ to $T_i$ for a basis $\{e_1,...,e_d\}$ of $\N^d$. By \cite[Theorem 3.3.17.(3)]{dllz2023logadic} and the local description of $\Omega_X(\log D),$ we may further assume that $X=\mathbb D^d.$ As both sides are coherent sheaves, it is enough to check the isomorphism at global sections.
		
		According to \cite[Proposition 3.2.25]{dllz2023logadic} (taking $P=\varnothing$ and $Q=\N^d=\bigoplus_{i=1}^d\N e_i$), we have $$\Omega_X^{\log}(X)\simeq \bigoplus_{i=1}^dL\left\langle x_1,...,x_d\right\rangle \delta(e_i),$$ where $\delta(e_i):=\phi(0,1\otimes e_i)$. Then $\phi$ sends $\delta(e_i)$ to $x_i^{-1}dx_i$. This completes the proof.
	\end{proof}
	
	Now we endow $X$ with the log structure defined by the open immersion $U \hookrightarrow X.$ Since $X$ is log smooth (e.g. see \cite[Example 3.1.13]{dllz2023logadic}), $\Omega_X^{i}(\log D)$ is a coherent sheaf, and therefore has a natural structure of sheaves of $K$-Banach spaces, and we have a natural condensed sheaf of $\mathcal O_X$-modules $\underline{\Omega_X^{i}(\log D)}$. 
	
	We want to compare the logarithmic de Rham cohomology of $X$ and the de Rham cohomology of $U$. This leads to the following theorem.
	
	
	
	\begin{thm} \label{logdr}
		Suppose $X$ is a smooth rigid analytic variety over $L$,  $D \subset X$ is a strictly normal crossing divisor, and $U=X-D.$ Then we have (non-filtered) quasi-isomorphisms in $D(\Mod_L^{\cond})$: $$\rg(X,\underline{\Omega_X^{\bullet}(\log D}))\xrightarrow{\simeq} \rg_{\underline{\operatorname{dR}}} (U).$$
	\end{thm}
	
	\begin{proof}
		As both sides satisfy analytic descent, by the existence of tubular neighborhoods, by \cite[Theorem 1.18]{kiehl1967derham}, we may assume that $X=\Sp(A\left\langle x_1,...,x_d\right\rangle),$ where $A$ is a smooth affinoid $L$-algebra, and $D$ is defined by the equation $g:=x_1\cdots x_d=0$. We note that $U$ is a quasi-Stein space (see \cite[Definition 2.3]{kiehl1971} for the definition): $U$ can be expressed as the increasing union $\bigcup U_n$ where $U_n:=\Sp(A\left\langle x_1,...,x_d,p^n/g\right\rangle).$ As $\Omega_X^{\bullet}(\log D)$ (resp. $\Omega_U^{\bullet}$) are complexes of coherent sheaves over $X$ (resp. $U$), we are reduced to showing that we have a quasi-isomorphism in $D(\Mod_L^{\cond})$: $$\underline{\Omega_X^{\bullet}(\log D)(X)} \to \underline{\Omega_U^{\bullet}(U)}$$ by \cite[Lemma 5.6]{bosco2023padicproetalecohomologydrinfeld} and \cite[Lemma 5.9]{bosco2023padicproetalecohomologydrinfeld} (which also applies for quasi-Stein spaces). Here $\Omega_X^{\bullet}(\log D)(X)$ is a complex of $L$-Banach spaces, and $\Omega_U^{\bullet}(U)$ is a complex of $L$-Fr\'echet spaces.
		
		According to \cite[Theorem 2.3]{kiehl1967derham} (and its proof), $\Omega_X^{\bullet}(\log D)(X) \to {\Omega_U^{\bullet}(U)}$ is a quasi-isomorphism in $D(\Mod_L)$. Therefore, to show that the same statement is true for complexes of condensed abelian groups, by \cite[Lemma A.33]{bosco2023padicproetalecohomologydrinfeld} it suffices to show that for each $i \geq 0$, $\Omega_X^{i}(\log D)(X) \hookrightarrow {\Omega_U^{i}(U)}$ has closed image (then the map $\Omega_X^{\bullet}(\log D)(X) \to {\Omega_U^{\bullet}(U)}$ is injective, and the quotient is a complex of $L$-Fr\'echet spaces which is acyclic). It suffices to prove the claim for $i=0,1$.
		
		In fact, it suffices to prove the claim for $i=0$: as $\Omega_X(\log D)(X)$ is a projective $\mathcal O(X)$-module, it is a direct summand of $\mathcal O^r(X)$ for some $r$, then $\Omega_X(U_n)$ is the direct summand of $\mathcal O^r(U_n)$. Then $\Omega_X(U)=\varprojlim_n \Omega_X(U_n)$ is the direct summand of $\mathcal O^r(U)=\varprojlim_n \mathcal O(U_n)$. If we know that $\mathcal O(X) \to \mathcal O(U)$ has closed image, then $\mathcal O^r(X) \to \mathcal O^r(U)$ also has closed image, therefore $\Omega_X(\log D)(X) \to \Omega_X(U)$ has closed image by projections.
		
		We are left to show that $\mathcal O(X) \to \mathcal O(U)$ has closed image. We have a chain of inclusions $$\mathcal O(X)=A\left\langle x_1,...,x_d\right\rangle \subset \mathcal O(U) \subset \cdots \subset \mathcal O(U_n)\cdots \subset  \mathcal O(U_0)=A\left\langle x_1,...,x_d,\frac{1}{x_1\cdots x_d}\right\rangle.$$ To prove the claim, it suffices to show that if $f \in \mathcal O(U)$ and $f \notin \mathcal O(X)$, then $f$ is contained in an open subset of $\mathcal O(U_0)$ that does not intersect with $\mathcal O(X)$. In fact, we can write $f$ as $f=\sum_{w\in \Z^d}a_w\underline{x}^w$, where we write $\underline{x}^w:=x_1^{w_1}\cdots x_d^{w_d}$ for $w=(w_1,...,w_d) \in \Z^d$, and $a_w \in A$. $f \notin \mathcal O(X)$ means there exists some $w \notin \N^d$ such that $a_w \neq 0.$ Then for $N$ sufficiently large, $(f+p^N\mathcal O^+(U_0)) \cap \mathcal O(X)=\varnothing.$ This completes the proof.
	\end{proof}
	
	\begin{rem}
		While we restrict to the case that $D \subset X$ is a strictly normal crossing divisor, Theorem \ref{logdr} also holds for any log smooth rigid analytic variety over $L$, i.e., if $X$ is log smooth over $L$ and $U:=X_{\operatorname{tr}}$ is the trivial locus of $X$ (i.e., the largest open where the induced log structure is trivial), then we have (non-filtered) quasi-isomorphisms in $D(\Mod_L^{\cond})$: $$\rg(X,\underline{\Omega_X^{\log,\bullet}})\xrightarrow{\simeq} \rg_{\underline{\operatorname{dR}}} (U).$$
		The proof is similar to the case for schemes: we can perform log-blow-ups so that $X-U$ is a strictly normal crossing divisor, and doing log-blow-ups will not change the logarithmic de Rham cohomology. However, as several ingredients are missing in the context of rigid analytic varieties, we will deal with this problem in our forthcoming works \cite{xslogdescent}.
	\end{rem}

	\section{Period sheaves} \label{period}
	
	We review the definition of period sheaves for log adic space and their local properties, which follow from \cite{dllz2023logrh}.
	
	Suppose $X$ is an analytic fs log adic space over $\Spa(\Q_p,\Z_p).$
	
	\subsection{Definitions}
	
	We begin with some basic definitions.
	
	\begin{defn}
		The following are defined to be sheaves on $X_{\proket}$ (see \cite[Definition 5.1.2]{dllz2023logadic} for the definition):
		
		(i) We define $\mathbb A_{\inf}:=W(\widehat{\mathcal{O}}^+_{X^{\flat}_{\proket}})$ and $\mathbb B_{\inf}:=\mathbb A_{\inf}\left[1/p\right],$ where the latter is equipped with a natural map $\theta:\mathbb B_{\inf} \to \widehat{\mathcal{O}}_{X_{\proket}}.$
		
		(ii) We define the positive de Rham sheaf $\mathbb B_{\operatorname{dR}}^+:=\lim_{n\in \N}B_{\inf,X}/(\ker \theta)^n$ with filtration given by $\fil^r \mathbb B_{\operatorname{dR}}^+:=(\ker \theta)^r\mathbb B_{\operatorname{dR}}^+.$
		
		(iii) We define the de Rham sheaf $\mathbb B_{\operatorname{dR}}:=\mathbb B_{\operatorname{dR}}^+\left[1/t\right],$ where $t$ is a generator of $\fil^1 \mathbb B_{\operatorname{dR}}^+.$ The filtration of $\mathbb B_{\operatorname{dR}}$ is given by $\fil^r \mathbb B_{\operatorname{dR}}:=\sum_{j\in \Z,r+j \geq 0}t^{-j}(\ker \theta)^{r+j}\mathbb B_{\operatorname{dR}}^+.$
	\end{defn}
	
	The following proposition shows that pro-Kummer \'etale locally these period sheaves are given by the relative period rings.
	
	\begin{prop} \cite[Proposition 2.2.4]{dllz2023logrh} \label{vanishing}
		Suppose $U \in X_{\proket}$ is a log affinoid perfectoid object with associated perfectoid space $\widehat{U}=\Spa(R,R^+)$ (see \cite[Remark 5.3.5]{dllz2023logadic} for the definition), then 
		
		(i) For $\f \in \{\widehat{\mathcal O}^+,\widehat{\mathcal O}^{\flat,+},\mathbb A_{\inf}, \mathbb B_{\inf}, \mathbb B_{\operatorname{dR}}^+, \mathbb B_{\operatorname{dR}}\}$, we have $\f({U})=\f (R,R^+).$
		
		(ii) $H^j(U,\mathbb B_{\operatorname{dR}}^+)=0$ and $H^j(U,\mathbb B_{\operatorname{dR}})=0$ for $j>0.$
	\end{prop}
	
	All these period rings carry a natural condensed sheaf structure defined as follows: For a given $X$ an analytic log adic space over $\Spa(\Q_p,\Z_p)$ and $\f \in \{\widehat{\mathcal O}^+,\widehat{\mathcal O}^{\flat,+},\mathbb A_{\inf}, \mathbb B_{\inf}, \mathbb B_{\operatorname{dR}}^+, \mathbb B_{\operatorname{dR}}\}$, denote $\mathcal{B}$ the basis of $X_{\proket}$ consisting of log affinoid perfectoid objects $U \in X_{\proket},$ then for any $U \in \mathcal B$, $\f(U)$ carries a natural structure of topological abelian groups, by giving $\widehat{\mathcal O}^+(U)$ the $p$-adic topology, and endowing all other period sheaves with induced topology, as explained in \cite[Corollary 6.6]{scholze2013p}. We can then associate to $\f$ a condensed abelian sheaf $\underline{\f}$ on $X_{\proket}$, such that $\underline{\f}(U)=\underline{\f(U)}$ for all log affinoid perfectoid objects $U \in \mathcal B.$ 
	
	We pass now to condensed pro-Kummer \'etale cohomology. We begin with the definition of condensed pro-Kummer \'etale cohomology. 
	
	\begin{defn}
		Let $X$ be an analytic log adic space over $\Spa(\Q_p,\Z_p).$ Denote by $*_{\kappa-\proet}$ the site of $\kappa$-small profinite sets with coverings given by finite families of jointly surjective maps. There is a natural morphism of sites $$f_{\operatorname{cond}}:X_{\proket}\to *_{\kappa-\proet}$$ defined by sending a profinite set $S$ to $X\times S \in X_{\proket}$. For a sheaf $\f$ of abelian groups on $X_{\proket},$ we define the condensed pro-Kummer \'etale cohomology of $\f$ by $$\rg_{\underline{\proket}}(X,\f):=Rf_{\operatorname{cond},*}\f.$$ Denote by $H^i_{\underline{\proket}}(X,\f):=R^if_{\operatorname{cond},*}\f$ its $i$-th cohomology group valued in $\operatorname{CondAb}.$
	\end{defn}
	
	\begin{rem}
		When $X$ is an analytic log adic space over $\Spa(C,\mathcal{O}_C)$, denote by $f:X_{\proket}\to \Spa(C,\mathcal{O}_C)_{\proet}$ the morphism of sites, then the above definition coincides with the derived pushforward of $f$, i.e. we have $\rg_{\underline{\proket}}(X,\f)=Rf_{*}\f.$
	\end{rem}
	
	\begin{prop} \label{contfct}
		Let $X$ be an analytic log adic space over $\Spa(C,\mathcal{O}_C)$ and $\f \in \{\widehat{\mathcal O}^+,\widehat{\mathcal O}^{\flat,+},\mathbb A_{\inf}, \mathbb B_{\inf},$ $\mathbb B_{\operatorname{dR}}^+, \mathbb B_{\operatorname{dR}}\}$. For any log affinoid perfectoid object $U \in \mathcal B$ and profinite set $S \in *_{\kappa-\proet}$, we have $$\f(U\times S)=\mathscr C^0(S,\f(U)).$$ In particular, we have $$H^i_{\underline{\proket}}(U,\f)=H^i_{{\proket}}(U,\underline{\f})$$ for all $i \geq 0.$
	\end{prop}
	
	\begin{proof}
		The proof of \cite[Corollary 6.6]{scholze2013p} or \cite[Corollary 4.9]{bosco2023padicproetalecohomologydrinfeld} goes through.
	\end{proof}
	
	Now let $X$ be a locally Noetherian fs log adic space over $\Spa(K,\mathcal O_K)$. We recall the definition of $\mathcal O\mathbb B_{\operatorname{dR,log}},$ a log version of the geometric de Rham period sheaf $\mathcal O\mathbb B_{\operatorname{dR}}.$ This notation and these properties will be used in the following sections.
	
	Let $U=\lim_{i\in I}U_i \in X_{\proket}$ be a log affinoid perfectoid object with $U_i=(\Spa(R_i,R_i^+),\M_i,\alpha_i)$ for each $i\in I$ and the associated perfectoid space $\widehat{U}=\Spa(R_{\infty},R_{\infty}^+),$ where $(R_{\infty},R^+_{\infty})$ is the $p$-adic completion of $\colim_{i\in I}(R_i,R_i^+),$ which is perfectoid. For each $i \in I$, write $M_i:=\M_i(U_i).$ By \cite[Theorem 5.4.3]{dllz2023logadic}, $(\widehat{\mathcal O}(U),\widehat{\mathcal O}^+(U))=(R_{\infty},R^+_{\infty})$, and the tilt of $(\widehat{\mathcal O}(U),\widehat{\mathcal O}^+(U))$ is\ $(R^{\flat +}_{\infty},R^{\flat}_{\infty})$. Define $M:=\M_{X_{\proket}}(U)=\lim_{i\in I}M_i,$ $M^{\flat}:=\M_{X^{\flat}_{\proket}}(U)=\lim_{a\mapsto a^p}M,$ and $\alpha^{\flat}:M^{\flat}\to \widehat{\mathcal O}^{\flat}$ the induced map. 
	
	\begin{defn}
		(i) The period sheaf $\mathcal O\mathbb B_{\operatorname{dR,log}}^+$ is defined to be the sheaf associated to the presheaf $$U \mapsto \varinjlim_{i\in I}\left(R_i^+\widehat{\otimes}_{W(k)}\mathbb A_{\operatorname{inf}}(U)\right)\left[1/p\right]\left[\cfrac{\alpha_i(m_i)}{[\alpha^{\flat}(m^{\flat})]},(m_i,m^{\flat})\in M_i \times_MM^{\flat}\right]^{\wedge_{\ker\theta_{\log}}}.$$
		Here $\widehat{\otimes}$ is the $p$-adic completion of the tensor product, $[\alpha^{\flat}(m^{\flat})]$ is obtained by the multiplicative map $$R^{+\flat}\to W(R^{+\flat})[1/p]/\xi^r:f \mapsto [f]$$ induced by $R^{+\flat}\to W(R^{+\flat})$, for each $r \geq 1$, and $$\theta_{\log}:\left(R_i^+\widehat{\otimes}_{W(k)}\mathbb A_{\operatorname{inf}}(U)\right)\left[1/p\right]\left[\cfrac{\alpha_i(m_i)}{[\alpha^{\flat}(m^{\flat})]},(m_i,m^{\flat})\in M_i \times_MM^{\flat}\right]\to R_{\infty}$$
		is induced by the map $R_i^+\to R^+_{\infty}$, $\theta:A_{\operatorname{inf}}(U)\to R^+_{\infty}$ such that $\theta_{\log}\left(\dfrac{\alpha_i(m_i)}{[\alpha^{\flat}(m^{\flat})]}\right)=1$ for any $(m_i,m^{\flat})\in M_i \times_MM^{\flat}$. In above, ``$\wedge_{\ker\theta_{\log}}$" means the completion with respect to the $\ker\theta_{\log}$-adic topology, and we take the direct limit after the completion. We equip $\mathcal O\mathbb B_{\operatorname{dR,log}}^+$ with the filtration $\fil^r\mathcal O\mathbb B_{\operatorname{dR,log}}^+:=(\ker \theta_{\log})^r\mathcal O\mathbb B_{\operatorname{dR,log}}^+.$
		
		(ii) The period sheaf $\mathcal O\mathbb B_{\operatorname{dR,log}}$ is the completion of the sheaf $\mathcal O\mathbb B_{\operatorname{dR,log}}^+\left[1/t\right]$  with respect to the filtration defined by $\fil^r\mathcal O\mathbb B_{\operatorname{dR,log}}^+\left[1/t\right]:=\sum_{j\in\Z,r+j\geq 0}t^{-j}(\ker \theta_{\log})^{r+j}\mathcal O\mathbb B_{\operatorname{dR,log}}^+.$
	\end{defn}
	
	There is a natural log connection $$\nabla:\mathcal O\mathbb B_{\operatorname{dR,log}}^+ \to \mathcal O\mathbb B_{\operatorname{dR,log}}^+ \otimes_{\mathcal O_{X_{\proket}}}\Omega_X^{\log}$$ on $\mathcal O\mathbb B_{\operatorname{dR,log}}^+$ which is defined as follows: By abuse of notation, we still denote by $\Omega_X^{\log}$ the pullback of log differential sheaf $\Omega_X^{\log}$ along the morphism of sites $X_{\proket} \to X_{\et}$, and we may write $\mathcal O\mathbb B_{\operatorname{dR,log}}^+(U)$ as $$\mathcal O\mathbb B_{\operatorname{dR,log}}^+(U)=\varinjlim_{i\in I}\widehat{S_i}=\varinjlim_{i\in I}\varprojlim_{r,s}(S_{i,r}/(\ker \theta_{\log})^s).$$ Here $$S_{i,r}=\left(R_i^+\widehat{\otimes}_{W(k)}(\mathbb A_{\operatorname{inf}}(U))/\xi^r\right)\left[1/p\right]\left[\cfrac{\alpha_i(m_i)}{[\alpha^{\flat}(m^{\flat})]},(m_i,m^{\flat})\in M_i \times_MM^{\flat}\right],$$ and $\widehat{S_i}:=\varprojlim_{r,s}(S_{i,r}/(\ker \theta_{\log})^s).$ Then there is a unique $\mathbb{B^+_{\operatorname{dR}}}(U)/\xi^r$-linear log connection $$\nabla_{i,r}:S_{i,r}\to S_{i,r}\otimes_{R_i}\Omega_X^{\log}(U_i)$$ extending derivations $d:R_i \to \Omega_X^{\log}(U_i)$ and $\delta:M_i \to \Omega_X^{\log}(U_i)$ such that $$\nabla_{i,r}\left(\dfrac{\alpha_i(m_i)}{[\alpha^{\flat}(m^{\flat})]}\right)=\dfrac{\alpha_i(m_i)}{[\alpha^{\flat}(m^{\flat})]}\delta(m_i)$$ for any $(m_i,m^{\flat})\in M_i \times_MM^{\flat}.$ Then the definition of $\nabla_{i,r}$ gives $$\nabla_{i,r}\left((\ker \theta_{\log})^s\right) \subset (\ker \theta_{\log})^{s-1}\otimes_{R_i}\Omega_X^{\log}(U_i) $$ for all $s \geq 1.$ Then the $\mathbb B^+_{\operatorname{dR}}$-linear log connection $$\nabla:\mathcal O\mathbb B_{\operatorname{dR,log}}^+ \to \mathcal O\mathbb B_{\operatorname{dR,log}}^+ \otimes_{\mathcal O_{X_{\proket}}}\Omega_X^{\log}$$ is obtained by taking $$\nabla:=\varinjlim_{i\in I}\varprojlim_r\left(\nabla_{i,r}^{\wedge_{\ker\theta_{\log}}}\right).$$ By inverting $t$, $\nabla$ further extends to a $\mathbb{B_{\operatorname{dR}}}$-linear log connection $$\nabla:\mathcal O\mathbb B_{\operatorname{dR,log}} \to \mathcal O\mathbb B_{\operatorname{dR,log}} \otimes_{\mathcal O_{X_{\proket}}}\Omega_X^{\log},$$ satisfying $$\nabla\left(\fil^rO\mathbb B_{\operatorname{dR,log}}\right)\subset \left(\fil^{r-1}O\mathbb B_{\operatorname{dR,log}}\right)\otimes_{\mathcal O_{X_{\proket}}}\Omega_X^{\log}$$ for all $r\in \Z.$
	
	\subsection{Local study of $\mathbb B_{\operatorname{dR}}$ and $\mathcal{O}\mathbb B_{\operatorname{dR,log}}$}
	
	We briefly discuss the local properties of $\mathbb B_{\operatorname{dR}}$ and $\mathcal{O}\mathbb B_{\operatorname{dR,log}}$, as studied in \cite[2.3]{dllz2023logrh}. Let $P$ be a toric monoid (i.e., an fs sharp monoid). Let $$\mathbb E:=\Spa(K\langle P\rangle,\mathcal O_K\langle P\rangle)$$ with the log structure induced by the natural map $P \to K\langle P\rangle$. The log adic space $\mathbb E$ admits a pro-Kummer \'etale cover as follows. For each $m \in \Z_{>0},$ let $\dfrac{1}{m}P$ be the toric monoid such that $P \hookrightarrow \dfrac{1}{m}P$ can be identified with the multiple map $[m]:P\to P.$ Let $P_{\Q_{\geq 0}}:=\colim_m \left(\dfrac{1}{m}P\right)$ and $P_{\Q}^{\operatorname{gp}}:=\left(P_{\Q_{\geq 0}}\right)^{\operatorname{gp}}\simeq P^{\operatorname{gp}}\otimes \Q.$ Let $$\mathbb E_{m,C}:=\Spa\left(C\left\langle\dfrac{1}{m}P\right\rangle,\mathcal O_C\left\langle\dfrac{1}{m}P\right\rangle\right)$$ equipped with the log structure induced by $\dfrac{1}{m}P \to K\langle\dfrac{1}{m}P\rangle.$ The morphism $\mathbb E_{m,C} \to \mathbb E_C$ is a finite Kummer \'etale cover with Galois group $$\Gamma_m=\operatorname{Hom}(P^{\operatorname{gp}},\boldsymbol{\mu}_m),$$ where $\boldsymbol{\mu}_m$ is the group of the $m$-th roots of unity in $C$. Denote the log affinoid perfectoid object by $$\widetilde{\mathbb E}_{C}:=\lim_m \mathbb E_{m,C} \in \mathbb E_{C,\proket},$$ where for $m|m'$, the transition maps are induced by $\dfrac{1}{m}P \hookrightarrow \dfrac{1}{m'}P.$ Then we have the associated perfectoid space $$\widehat{\widetilde{\mathbb E}}_C:=\Spa(C\langle P_{\Q_{\geq 0}}\rangle,\mathcal O_C\langle P_{\Q_{\geq 0}}\rangle).$$ The morphism $\widehat{\widetilde{\mathbb E}}_C \to \mathbb E_C$ is a Galois pro-finite Kummer \'etale cover with Galois group $$\Gamma=\operatorname{Hom}(P^{\operatorname{gp}}_{\Q}/P^{\operatorname{gp}},\boldsymbol{\mu}_{\infty}),$$ where $\boldsymbol{\mu}_{\infty}:=\cup_{m\in \N} \boldsymbol{\mu}_{m}.$ The natural action of $\Gamma$ on $\mathcal O_C\langle P_{\Q_{\geq 0}}\rangle$ is given by $$\gamma(T^a)=\gamma(a)T^a$$ for all $\gamma \in \Gamma$ and $a \in P_{\Q_{\geq 0}},$ where $T^a$ is the corresponding element of $a$ in $\mathcal O_C\langle P_{\Q_{\geq 0}}\rangle$.
	
	Suppose $X=\Spa(R,R^+)$ is an affinoid fs log adic space of finite type over $\Spa(K,\mathcal O_K)$ with a strictly \'etale morphism $X \to \mathbb E$, which can be written as a composite of rational embedding and finite \'etale maps. Pulling back $\widehat{\widetilde{\mathbb E}}_C \to \mathbb E$ along $X \to \mathbb E$, we get $$\widetilde{X}_C:=X_C\times_{\mathbb E_C}\widehat{\widetilde{\mathbb E}}_C=\Spa(R_{\infty},R^+_{\infty}),$$ which is  a Galois pro-finite Kummer \'etale cover with Galois group $\Gamma$.
	
	Let $\mathbb{B_{\operatorname{dR}}}|_{\widetilde{X}_C}[[P]]$ be the sheaf of monoid algebras. For $a\in P,$ denote by $e^a$ the image of $a$ in $\mathbb{B_{\operatorname{dR}}}|_{\widetilde{X}_C}[[P]]$ via the natural map $P \to \mathbb{B_{\operatorname{dR}}}|_{\widetilde{X}_C}[[P]]:a \mapsto e^a.$ Let $\mathfrak m \subset \mathbb{B_{\operatorname{dR}}}|_{\widetilde{X}_C}[P]$ be the sheaf of ideals generated by $\{e^a-1\}_{a\in P},$ and denote by $$\mathbb{B_{\operatorname{dR}}}|_{\widetilde{X}_C}[[P-1]]:=\varprojlim_r(\mathbb{B_{\operatorname{dR}}}|_{\widetilde{X}_C}[[P]]/\mathfrak m^r).$$
	Consider the monoid homomorphism (with respect to the additive structure on $\mathbb{B_{\operatorname{dR}}}|_{\widetilde{X}_C}[[P-1]]$) $$P \to \mathbb{B_{\operatorname{dR}}} |_{\widetilde{X}_C}[[P-1]] :a\mapsto \log(e^a) :=\sum_{l=1}^{\infty}(-1)^{l-1}\dfrac{1}{l}(e^a-1)^l,$$ which uniquely extends to a group homomorphism $$P^{\operatorname{gp}} \to \mathbb{B_{\operatorname{dR}}} |_{\widetilde{X}_C}[[P-1]] :a\mapsto y_a:=\log(e^{a^+})-\log(e^{a^-}),$$ where $a=a^+-a^-$ for $a^+,a^- \in P.$ Choose a $\Z$-basis $\{a_1,...,a_n\}$ of $P^{\operatorname{gp}}$, and for each $j=1,2,...,n,$ denote by $y_j:=y_{a_j},$ then we have a canonical isomorphism of $\mathbb{B_{\operatorname{dR}}}|_{\widetilde{X}_C}$-algebras: $$\mathbb{B_{\operatorname{dR}}}|_{\widetilde{X}_C}[[y_1,...,y_n]]\xrightarrow{\simeq}\mathbb{B_{\operatorname{dR}}}|_{\widetilde{X}_C}[[P-1]]:y_j \mapsto y_j:=y_{a_j},$$ matching the ideals $(y_1,..,y_n)^r$ and $(\xi,y_1,..,y_n)^r$ of the source with the ideals $\mathfrak{m}^r$ and $(\xi,\mathfrak{m})^r$ of the target, respectively, for all $r \in \Z.$
	
	Now, similar to \cite[Proposition 6.10]{scholze2013p}, we can give a local description of $\mathcal{O}\mathbb B_{\operatorname{dR,log}}^+.$ Recall that we denote by $U=\varprojlim_{i\in I}U_i \in X_{\proket}/\widetilde{X}_C$ a log affinoid perfectoid object with $U_i=(\Spa(R_i,R_i^+),\M_i,\alpha_i).$ Denote by $\alpha^{\flat}:\mathcal M^{\flat}:=\lim_{a\mapsto a^p}\mathcal M \to \widehat{\mathcal O}$ the induced map. Consider the map (note that the sheaf $\mathcal M^{\flat}|_{\widetilde{X}_C}$ is generated by $P_{\Q_{\geq 0}}$ and $\varprojlim_{f\mapsto f^p}\mathcal O^{\times}_{X_{\proket}}|_{\widetilde{X}_C}$, therefore $\alpha^{\flat}(a)$ makes sense) $$(\mathbb B_{\operatorname{dR}}^+(U)/\xi^r)[P]\to S_{i,r}:e^a \mapsto \dfrac{\alpha(a)}{[\alpha^{\flat}(a)]}, \text{ for all } a \in P,$$ which sends $(\xi,\mathfrak{m})$ to $\ker(\theta_{\log}),$ so this map induces a map $\mathbb{B_{\operatorname{dR}}}|_{\widetilde{X}_C}[[P-1]] \to \widehat{S_i}.$ After taking completion and sheafification, we obtain a map
	\begin{equation} \label{obdrlocal}
		\mathbb B^{+}_{\operatorname{dR}}|_{\widetilde{X}_C}[[P-1]] \to \mathcal O\mathbb B_{\operatorname{dR,log}}^+|_{\widetilde{X}_C}
	\end{equation}
	on $X_{\proket}/\widetilde{X}_C.$ The map is compatible with filtrations on both sides, where the filtration on $\mathbb B^{+}_{\operatorname{dR}}|_{\widetilde{X}_C}[[P-1]]$ is given by $\fil^r \mathbb B^{+}_{\operatorname{dR}}|_{\widetilde{X}_C}[[P-1]]:=(\xi,\mathfrak{m})^r\mathbb B^{+}_{\operatorname{dR}}|_{\widetilde{X}_C}[[P-1]]$ for all $r \in\Z.$ We have the following proposition.
	
	\begin{prop}\cite[Proposition 2.3.15]{dllz2023logrh} \label{obdrlocalprop}
		The map (\ref{obdrlocal}) is an isomorphism of filtered sheaves.
	\end{prop}
	
	\begin{rem}
		In fact, for $U=\varprojlim_{i\in I}U_i$ as above, the natural map $\widehat{S_i}\to \mathcal O\mathbb B_{\operatorname{dR,log}}^+(U)$ is already an isomorphism.
	\end{rem}
	
	\begin{rem} \label{obdrconn}
		We can also describe the log connection of $\mathcal O\mathbb B_{\operatorname{dR,log}}^+$ using this isomorphism. In fact, choose a $\Z$-basis $\{a_1,...,a_n\}$ of $P^{\operatorname{gp}}$, write $a_j=a_j^+-a_j^-$ for $j=1,2,...,n,$ then $y_j=b_j^+-b_j^-:=\log\left(e^{a^+_j}\right)-\log\left(e^{a^-_j}\right),$ and the isomorphism (\ref{obdrlocal}) sends $y_j$ to $\dfrac{\alpha(b_j^+)}{[\alpha^{\flat}(b_j^+)]}-\dfrac{\alpha(b_j^-)}{[\alpha^{\flat}(b_j^-)]}.$ Therefore, $$\nabla(y_j)=\nabla\left(\dfrac{\alpha(b_j^+)}{[\alpha^{\flat}(b_j^+)]}\right)-\nabla\left(\dfrac{\alpha(b_j^-)}{[\alpha^{\flat}(b_j^-)]}\right)=\delta(a_j^+)-\delta(a_j^-)=\delta(a_j).$$ We note that, over $\widetilde{X}_C,$ the log connection of $\mathcal O\mathbb B_{\operatorname{dR,log}}^+$ is compatible with the canonical one of the polynomial algebra $\mathbb{B_{\operatorname{dR}}}|_{\widetilde{X}_C}[[y_1,...,y_n]]$, since $\Omega_X^{\log} \simeq \bigoplus_{j=1}^n\mathcal O_X\delta(a_j)$ by \cite[Theorem 3.3.17, Corollary 3.3.18, Proposition 3.2.25, and Corollary 3.2.29]{dllz2023logadic}.
	\end{rem}
	
	\begin{cor}\cite[Corollary 2.3.17]{dllz2023logadic} \label{obdrlocalcor}
		The isomorphism \ref{obdrlocal} induces isomorphisms $$\fil^r \mathcal O\mathbb B_{\operatorname{dR,log}}^+ \simeq t^r \mathbb B_{\operatorname{dR}}^+\{W_1,...,W_n\}$$ over $X_{\proket}/\widetilde{X}_C,$ for all $r \in \Z.$ Here $\mathbb B_{\operatorname{dR}}^+\{W_1,...,W_n\}$ is the ring of power series that are $t$-adically convergent, and $$W_j=t^{-1}y_j$$ for each $1 \leq j \leq n$. In particular, $$\gr^r \mathcal O\mathbb B_{\operatorname{dR,log}} \simeq t^r \widehat{\mathcal O}_{X_{\proket}}[W_1,...,W_n],$$ for all $r \in \Z.$
	\end{cor}
	
	The above discussion allows us to deduce the Poincar\'e lemma for $\mathcal{O}\mathbb B_{\operatorname{dR,log}}^+$ and $\mathcal{O}\mathbb B_{\operatorname{dR,log}}$ with log connections.
	
	\begin{prop}\cite[Corollary 2.4.2]{dllz2023logrh}
		Let $X$ be a log smooth rigid analytic variety with a fine log structure defined over $K$ of dimension $n$. Then we have an exact sequence of sheaves on $X_{\proket}:$ $$0\to \mathbb B_{\operatorname{dR}}^+ \to \mathcal{O}\mathbb B_{\operatorname{dR,log}}^+ \xrightarrow{\nabla} \mathcal{O}\mathbb B_{\operatorname{dR,log}}^+\otimes_{\mathcal O_X}\Omega^{\operatorname{log},1}_X \xrightarrow{\nabla} \cdots \xrightarrow{\nabla} \mathcal{O}\mathbb B_{\operatorname{dR,log}}^+\otimes_{\mathcal O_X}\Omega^{\operatorname{log},n}_X \to 0. $$ The exact sequence of sheaves also holds if we replace $\mathbb B_{\operatorname{dR}}^+$ and $\mathcal{O}\mathbb B_{\operatorname{dR,log}}^+$ with $\mathbb B_{\operatorname{dR}}$ and $\mathcal{O}\mathbb B_{\operatorname{dR,log}}$, respectively. For $r \in \Z$ we also have compatible exact sequences of sheaves on $X_{\proket}:$ $$0\to \fil^r\mathbb B_{\operatorname{dR}} \to \fil^r\mathcal{O}\mathbb B_{\operatorname{dR,log}} \xrightarrow{\nabla} (\fil^{r-1}\mathcal{O}\mathbb B_{\operatorname{dR,log}})\otimes_{\mathcal O_X}\Omega^{\operatorname{log},1}_X \xrightarrow{\nabla} \cdots \xrightarrow{\nabla} (\fil^{r-n}\mathcal{O}\mathbb B_{\operatorname{dR,log}})\otimes_{\mathcal O_X}\Omega^{\operatorname{log},n}_X \to 0.$$
	\end{prop}
	
	
	\section{The pro-Kummer \'etale cohomology of $\mathbb B_{\operatorname{dR}}^+$ and $\mathbb B_{\operatorname{dR}}$}
	
	We will establish an analog version of \cite[Theorem 1.8]{bosco2023padicproetalecohomologydrinfeld} for log adic spaces. Moreover, we will also construct a fully faithful functor from the category of filtered $ \mathcal{O}_X $-modules with integrable connections to the category of $ \mathbb{B}_{\operatorname{dR}} $-local systems, thereby generalizing \cite[Theorem 7.6]{scholze2013p}. 
	
	We start with the definition of local systems for log adic spaces. Let $X$ be a log smooth rigid analytic variety over $K$, equipped with a fine log structure.
	
	\begin{defn}
		(i) A $\mathbb B_{\operatorname{{dR}}}^+$-local system is a sheaf of $\mathbb B_{\operatorname{{dR}}}^+$-modules $\mathbb M^+$ that is locally on $X_{\proket}$ free of finite rank. 
		
		(ii) An $\mathcal O \mathbb B_{\operatorname{{dR,log}}}^+$-module with integrable log connection is a sheaf of $\mathcal O \mathbb B_{\operatorname{{dR,log}}}^+$-modules $\mathcal M$ that is locally on $X_{\proket}$ free of finite rank, together with an integrable log connection $$\nabla_{\mathcal M}:\mathcal M \to \mathcal M \otimes_{\mathcal{O}_X}\Omega_X^{\log},$$ satisfying the Leibniz rule with respect to the derivation $\nabla$ of $\mathcal O \mathbb B_{\operatorname{{dR,log}}}^+$. 
	\end{defn}
	
	\begin{theorem}
		The functor $$\mathbb M^+ \mapsto (\mathcal M, \nabla_{\mathcal M}):=\left(\mathbb M^+ \otimes_{\mathbb B_{\operatorname{dR}}^+}\mathcal O \mathbb B_{\operatorname{{dR,log}}}^+,\operatorname{id}\otimes \nabla\right)$$ gives an equivalence between the category of $\mathbb B_{\operatorname{{dR}}}^+$-local systems and the category of $\mathcal O \mathbb B_{\operatorname{{dR,log}}}^+$-modules with integrable log connection. The inverse is given by $(\mathcal M, \nabla_{\mathcal M})\mapsto \mathbb M^+ :=\left(\mathcal M \right)^{\nabla_{\mathcal M}=0}.$
	\end{theorem}
	
	\begin{proof}
		It is clear that for a $\mathbb B_{\operatorname{{dR}}}^+$-local system $\mathbb M^+,$ $$\left(\mathbb M^+ \otimes_{\mathbb B_{\operatorname{dR}}^+}\mathcal O \mathbb B_{\operatorname{{dR,log}}}^+\right)^{\nabla=0}=\mathbb M^+.$$ We need to show that for $(\mathcal M,\nabla_M)$ a $\mathcal O \mathbb B_{\operatorname{{dR,log}}}^+$-module with integrable log connection, the natural morphism $$ \left(\mathcal M \right)^{\nabla_{\mathcal M}=0}\otimes_{\mathbb B_{\operatorname{dR}}^+}\mathcal O  \mathbb B_{\operatorname{{dR,log}}}^+ \xrightarrow{ \simeq} \mathcal M, $$ is an isomorphism. The problem is local, so we can use Proposition \ref{obdrlocalprop} and reduce to the case that $X=\Spa(R,R^+) \to \mathbb E$ is a strictly \'etale morphism, which can be written as a composite of rational embedding and finite \'etale maps. Now suppose $M$ is a locally free $\mathbb B_{\operatorname{dR}}^+(R^{\infty})[[y_1,...,y_n]]$-module with an integrable log connection $\nabla_M$, using Remark \ref{obdrconn} and \cite[Proposition 8.9]{katz1970connection}, we have $M \simeq M^{\nabla=0}\otimes_{\mathbb B_{\operatorname{dR}}^+(R^{\infty})}\mathbb B_{\operatorname{dR}}^+(R^{\infty})[[y_1,...,y_n]],$ which concludes the proof.
	\end{proof}
	
	\begin{defn}\cite[Definition 7.4 and Definition 7.5]{scholze2013p}
		(i) A filtered $\mathcal O_X$-module with integrable log connection is a locally free $\mathcal O_X$-module $\mathcal{E}$ on $X$ (with \'etale, pro-\'etale or analytic topology, since they are all equivalent), together with a separated and exhaustive decreasing filtration $\{\fil^r \mathcal{E}\}_{r \in \Z},$ by locally direct summands, and an integrable log connection $\nabla$ satisfying the Griffiths transversality.
		
		(ii) We say $\mathcal{E}$ and an $\mathcal O \mathbb B_{\operatorname{{dR,log}}}^+$-module with integrable log connection $\mathcal M$ are associated if there is an isomorphism of sheaves on $X_{\proket}$ $$\mathcal M \otimes_{\mathcal O \mathbb B_{\operatorname{{dR,log}}}^+}\mathcal O \mathbb B_{\operatorname{{dR,log}}}\simeq \mathcal E \otimes_{\mathcal{O}_X}O \mathbb B_{\operatorname{{dR,log}}},$$ compatible with filtrations and connections, where the filtration on the left side is the one compatible with that on $O\mathbb B_{\operatorname{dR,log}}$, i.e., $$\fil^r(\mathcal M \otimes_{O \mathbb B_{\operatorname{{dR,log}}}^+}\mathcal O \mathbb B_{\operatorname{{dR,log}}}):=\sum_{j\in\Z,r+j\geq 0}t^{-j}(\ker \theta_{\log})^{r+j}\mathcal M.$$
	\end{defn}
	
	\begin{thm}
		(i) If $\mathcal M$ is an $\mathcal O \mathbb B_{\operatorname{{dR,log}}}^+$-module with integrable log connection and horizontal section $\mathbb M^+,$ which is associated to a filtered $\mathcal O_X$-module with integrable connection $\mathcal E,$ then $$\mathbb M^+=\fil^0(\mathcal{E}\otimes_{\mathcal O_X}\mathcal{O}\mathbb B_{\operatorname{dR,log}})^{\nabla=0}.$$ Similarly, one can reconstruct $\mathcal E$ by $$\mathcal E_{\et}\simeq \nu_*\left(\mathbb M^+ \otimes_{\mathbb B_{\operatorname{dR}}^+}\mathcal O \mathbb B_{\operatorname{{dR,log}}}\right),$$ with the log connection on $\mathcal O \mathbb B_{\operatorname{{dR,log}}}$ and the filtration induced by those on $\mathbb M^+$ and $\mathcal O \mathbb B_{\operatorname{{dR,log}}}$. Here $\nu$ is the morphism of site $X_{\proket}\to X_{\et}$.
		
		(ii) If $\mathcal E$ is an $\mathcal O_X$-module with integrable log connection, the sheaf $$\mathbb M^+=\fil^0(\mathcal{E}\otimes_{\mathcal O_X}\mathcal{O}\mathbb B_{\operatorname{dR,log}})^{\nabla=0}$$ is a $\mathbb B_{\operatorname{{dR}}}^+$-local system such that $\mathcal E$ is associated to $\mathcal M=\mathbb M^+ \otimes_{\mathbb B_{\operatorname{dR}}^+}\mathcal O \mathbb B_{\operatorname{{dR,log}}}.$
		
		In particular, the functor sending $\mathcal E$ to $\mathbb M^+$ is a fully faithful functor from the category of filtered $\mathcal O_X$-module with integrable log connection to the category of $\mathbb B_{\operatorname{{dR}}}^+$-local systems.
	\end{thm}
	
	\begin{proof}
		The proof goes as the same way as the proof of \cite[Theorem 7.6]{scholze2013p}, by replacing $\mathcal{O}\mathbb B_{\operatorname{dR}}$ with $\mathcal{O}\mathbb B_{\operatorname{dR,log}}$ and $\Omega_X$ with $\Omega_X^{\log}.$
	\end{proof}
	
	\begin{defn}
		Let $X$ be a log smooth rigid-analytic variety over $K$ with fine log structure. Let $(\mathcal{E},\nabla,\fil^{\bullet})$ be a filtered $\mathcal O_X$-module with integrable log connection. We define the log de Rham complex associated to $(\mathcal E,\nabla)$ by $$\logdr_X^{\mathcal E}:=\left[\mathcal E \xrightarrow{\nabla} \mathcal E \otimes_{\mathcal O_X}\Omega^{\operatorname{log},1}_X \xrightarrow{\nabla} \mathcal E \otimes_{\mathcal O_X}\Omega^{\operatorname{log},2}_X \xrightarrow{\nabla}  \cdots   \right].$$ We equip $\logdr_X^{\mathcal E}$ with the filtration given by $$\fil^r\logdr_X^{\mathcal E}:=\left[\fil^r\mathcal E \xrightarrow{\nabla} \fil^{r-1}\mathcal E \otimes_{\mathcal O_X}\Omega^{\operatorname{log},1}_X \xrightarrow{\nabla} \fil^{r-2}\mathcal E \otimes_{\mathcal O_X}\Omega^{\operatorname{log},2}_X \xrightarrow{\nabla}  \cdots \right]$$ for $r \in \Z$.
		
		(i) For $i \geq 0$, the condensed de Rham cohomology group $H^i_{\underline{\operatorname{logdR}}}(X,\mathcal E)$ with coefficients is defined to be the $i$-th cohomology group of the complex $$\rg_{\underline{\operatorname{logdR}}}(X,\mathcal E):=\rg(X,\underline{\logdr_X^{\mathcal E}})$$ of $D(\Mod_K^{\cond})$.
		
		(ii) We define the complex in $D(\Mod_K^{\cond})$ $$\rg_{\underline{\operatorname{logdR}}}(X_{B_{\operatorname{dR}}},\mathcal E):=\rg(X,\underline{\logdr_X^{\mathcal E}}\otimes_K^{\blacksquare}\underline{B_{\operatorname{dR}}}),$$ and we endow it with the filtration induced from the tensor product filtration.
	\end{defn}

	We will prove the following theorem.
	
	\begin{thm} \label{bdrcoh}
		Let $X$ be a log smooth rigid analytic variety defined over $K$ with fine saturated log structure. Let $(\mathcal{E},\nabla,\fil^{\bullet})$ be a filtered $\mathcal{O}_X$-module with integrable log connection and associated $\mathbb B_{\operatorname{dR}}^+$-local system $\mathbb M^+$. Denote by $\mathbb M:=\mathbb M^+[1/t].$
		
		(i) We have a $\mathscr{G}_K$-equivariant, compatible with filtrations, natural quasi-isomorphism in $D(\Mod_K^{\cond})$: $$\rg_{\underline{\proket}}(X_C,\mathbb{M})\simeq \rg_{\underline{\operatorname{logdR}}}(X_{B_{\operatorname{dR}}},\mathcal{E}).$$
		
		(ii) Assume $X$ is connected and paracompact. Then, for each $r\in \Z$, we have a natural $\mathscr{G}_K$-equivariant quasi-isomorphisms in $D(\Mod_K^{\cond})$: $$\rg_{\underline{\proket}}(X_C,\fil^r\mathbb{M})\simeq \fil^r(\rg_{\underline{\operatorname{logdR}}}(X,\mathcal{E})\otimes_K^{\blacksquare}\underline{B_{\operatorname{dR}}}),$$ where $\rg_{\underline{\operatorname{logdR}}}(X)$ is the log de Rham cohomology complex in $D(\Mod_K^{\cond})$, and the filtration is induced by the tensor product filtration.
	\end{thm}

	To prove the theorem, we need the following proposition, which comes from the local study of $\mathbb B_{\operatorname{dR}}$ and $\mathcal{O}\mathbb B_{\operatorname{dR,log}}$, as in Section \ref{period} and we follow the notations there. Suppose that $X=\Spa(R,R^+)$ is an affinoid fs log adic space of finite type over $\Spa(K,\mathcal O_K)$ with a strictly \'etale morphism $X \to \mathbb E$, which can be written as a composite of rational embeddings and finite \'etale maps. Pulling back $\widehat{\widetilde{\mathbb E}}_C \to \mathbb E_C$ along $X \to \mathbb E$, we get $$\widetilde{X}_C:=X_C\times_{\mathbb E_C}\widehat{\widetilde{\mathbb E}}_C=\Spa(R_{\infty},R^+_{\infty}),$$ which is a perfectoid log affinoid adic space.
	
	\begin{prop} \cite[Lemma 6.11]{bosco2023padicproetalecohomologydrinfeld}
		Suppose $S \in *_{\kappa,\proet},$ then we have $$H^i_{\proket}(X_C\times S,\gr^j\mathcal{O}\mathbb B_{\operatorname{dR,log}})=\left\{\begin{aligned}
			& \mathscr C^0(S,R\widehat{\otimes}_KC(j)) & \text{ if }i=0  \\
			& 0 & \text{ if }i>0
		\end{aligned}\right.,$$ where $C(j)$ is the Tate twist of degree $j$ for any $j \in \Z$.
	\end{prop}
	
	\begin{proof}
		By twisting we can reduce to the case $j=0$. When $S=*,$ this is \cite[Lemma 3.3.15]{dllz2023logrh}. For the general case, since $\widehat{\widetilde{\mathbb E}}_C \to \mathbb E_C$ is a Galois pro-finite Kummer \'etale cover with Galois group $\Gamma$, the Cartan–Leray spectral sequence associated to the affinoid perfectoid $\Gamma$-cover $\widetilde{X}_C \times S\to X_C \times S$, combined with the vanishing theorem for $\gr^0\mathcal{O}\mathbb B_{\operatorname{dR,log}}$ (by Proposition \ref{vanishing} and Proposition \ref{contfct}) gives $$H^i_{\proket}(X_C\times S,\gr^0\mathcal{O}\mathbb B_{\operatorname{dR,log}})=H^i_{\operatorname{cont}}(\Gamma,\gr^0\mathcal{O}\mathbb B_{\operatorname{dR,log}}(\widetilde{X}_C\times S))$$ for $i \geq 0.$ Note that $\widetilde{X}_C\times S=\Spa(\mathscr C^0(S,R_{\infty}),\mathscr C^0(S,R^+_{\infty})).$ The proofs of \cite{lz2017rh} and \cite[Lemma 3.3.15]{dllz2023logrh} give $H^i_{\proket}(X_C\times S,\gr^0\mathcal{O}\mathbb B_{\operatorname{dR,log}})=0$ for $i>0$ and $$H^0_{\proket}(X_C\times S,\gr^0\mathcal{O}\mathbb B_{\operatorname{dR,log}})=\mathscr C^0(S,R) \widehat{\otimes}_KC.$$ The proposition then follows from the fact that $\mathscr C^0(S,R) \widehat{\otimes}_KC \simeq \mathscr C^0(S,R\widehat{\otimes}_KC),$ which follows from \cite[
		Corollary 10.5.4]{perez2010lcs}.
	\end{proof}
	
	\begin{prop} \label{bdrdirectimage}
		Let $X$ be a log smooth rigid analytic variety defined over $K$ with fine log structure. Let $(\mathcal{E},\nabla,\fil^{\bullet})$ be a filtered $\mathcal O_X$-module with integrable log connection and associated $\mathbb B_{\operatorname{dR}}^+$-local system $\mathbb M^+:=\fil^0(\mathcal{E}\otimes_{\mathcal O_X}\mathcal{O}\mathbb B_{\operatorname{dR,log}})^{\nabla=0},$ and $\mathbb M=\mathbb M^+[1/t].$ Denote by $\lambda$ the morphism of sites $$X_{\proket}/X_C \simeq X_{C,\proket} \to X_{C,\et,\cond}.$$ Then we have a natural quasi-isomorphism of complexes of sheaves valued in $D(\Mod_K^{\cond})$ on $X_{C,\et}$: $$(R\lambda_*\mathbb M)^{\blacktriangledown}\simeq \underline{\logdr}_X^{\mathcal E} \otimes_K^{\blacksquare}\underline{B_{\operatorname{dR}}},$$ where the notation $(-)^{\blacktriangledown}$ is defined in \cite[2.3]{bosco2023padicproetalecohomologydrinfeld}. Moreover, the quasi-isomorphism is compatible with filtrations, where the filtration on the right-hand side is given by the tensor product filtration.
	\end{prop} 
	
	\begin{proof}
		Suppose $X$ is of dimension $n$ and connected. Denote by $\mathcal M$ the $\mathcal O \mathbb B_{\operatorname{{dR,log}}}^+$-module with integrable log connection associated to $(\mathcal{E},\nabla,\fil^{\bullet})$ , and $$\mathcal M':=\mathcal M \otimes_{O \mathbb B_{\operatorname{{dR,log}}}^+}O \mathbb B_{\operatorname{{dR,log}}}\simeq \mathcal E \otimes_{\mathcal{O}_X}O \mathbb B_{\operatorname{{dR,log}}}.$$ The logarithmic Poincar\'e lemma gives an exact sequence of sheaves on $X_{C,\proket}$: $$0\to \mathbb M \to \mathcal M' \xrightarrow{\nabla} \mathcal M' \otimes_{\mathcal O_X}\Omega^{\operatorname{log},1}_X \xrightarrow{\nabla} \cdots \xrightarrow{\nabla} \mathcal M'\otimes_{\mathcal O_X}\Omega^{\operatorname{log},n}_X \to 0,$$ which is also exact after taking $\fil^r$ again by the logarithmic Poincar\'e lemma: $$0\to \fil^r\mathbb M \to \fil^r\mathcal{M'} \xrightarrow{\nabla} (\fil^{r-1}\mathcal{M'})\otimes_{\mathcal O_X}\Omega^{\operatorname{log},1}_X \xrightarrow{\nabla} \cdots \xrightarrow{\nabla} (\fil^{r-n}\mathcal{M'})\otimes_{\mathcal O_X}\Omega^{\operatorname{log},n}_X \to 0.$$ Therefore, we have a natural quasi-isomorphism between $(R\lambda_*\mathbb M)^{\blacktriangledown}$ and $$\mathcal C^{\bullet}:=R\lambda_*\left[\mathcal M' \xrightarrow{\nabla} \mathcal M' \otimes_{\mathcal O_X}\Omega^{\operatorname{log},1}_X \xrightarrow{\nabla} \cdots \xrightarrow{\nabla} \mathcal M'\otimes_{\mathcal O_X}\Omega^{\operatorname{log},n}_X \right]^{\blacktriangledown},$$ which is compatible with filtrations. 
		
		We claim we have a natural filtered quasi-isomorphism 
		\begin{equation} \label{naturalbdr}
			\underline{\logdr}_X^{\mathcal E} \otimes_K^{\blacksquare}\underline{B_{\operatorname{dR}}} \xrightarrow{\simeq} \mathcal C^{\bullet}
		\end{equation}
		of complexes of sheaves on $X_{C,\et}$ (as explained in the Lemma 2.5 of \cite{lz2017rh}, we can regard them as sheaves over $X_{C,\et}$). To construct this morphism, it suffices to define a natural morphism $\mathcal O_X \otimes_K^{\blacksquare} \underline{B_{\operatorname{dR}}} \to (\lambda_*\mathcal{O}\mathbb  B_{\operatorname{dR,log}})^{\blacktriangledown}$ of sheaves on $X_{C,\et},$ which is compatible with filtration. This is already constructed in the proof of \cite[Corollary 6.12]{bosco2023padicproetalecohomologydrinfeld} (also in \cite[Lemma 3.7]{lz2017rh} and \cite[Lemma 3.3.2]{dllz2023logrh}).
		
		As explained in \cite[Corollary 6.12]{bosco2023padicproetalecohomologydrinfeld}, since the filtrations on $\underline{\logdr}_X^{\mathcal E} \otimes_K^{\blacksquare} \underline{B_{\operatorname{dR}}}$ and $\mathcal C^{\bullet}$ are exhaustive and complete, by \cite[Lemma 5.2]{bms2}, it suffices to show the natural morphism (\ref{naturalbdr}) is a quasi-isomorphism on graded pieces, which means that for any $j \in \Z$, we have a natural isomorphism
		$$\gr^j(\underline{\logdr}_X^{\mathcal E} \otimes_K^{\blacksquare}\underline{B_{\operatorname{dR}}})\to \gr^j \mathcal C^{\bullet}$$
		on $X_{C,\et}.$ Since (\ref{naturalbdr}) is also compatible with respect to the naive filtration on both sides, it suffices to show that for any locally free $\mathcal O_X$-module $\f$ on $X_{\et}$ of finite rank, we have a natural quasi-isomorphism 
		\begin{equation} \label{gradeiso}
			\underline{\f} \otimes_K^{\blacksquare} \gr^j\underline{B_{\operatorname{dR}}} \xrightarrow{\simeq} R\lambda_*\left(\nu^*\f\otimes_{\mathcal{O}_X}\gr^j\mathcal{O}\mathbb B_{\operatorname{dR,log}}\right)^{\blacktriangledown}
		\end{equation}	
		on $X_{C,\et}$, for any $j \in \Z$, where $\nu:X_{\proket} \to X_{\et}$. Then we can check (\ref{gradeiso}) by checking graded pieces at each degree on both sides of (\ref{naturalbdr}). Write $\lambda:X_{C,\proket} \to X_{C,\et,\cond}$ as 
		$$X_{C,\proket} \xrightarrow{\alpha} X_{C,\proet} \xrightarrow{\lambda'} X_{C,\et,\cond},$$
		and write $\nu$ as $$X_{\proket} \to X_{\proet} \xrightarrow{\nu'} X_{\et}.$$ 
		According to the proof of \cite[Corollary 6.12]{bosco2023padicproetalecohomologydrinfeld}, we have a natural quasi-isomorphism $$\underline{\f} \otimes_K^{\blacksquare} \gr^j\underline{B_{\operatorname{dR}}} \xrightarrow{\simeq} R\lambda'_*\left(\nu'^*\f\otimes_{\mathcal{O}_X}\gr^j\mathcal{O}\mathbb B_{\operatorname{dR}}\right)^{\blacktriangledown}$$ on $X_{C,\et},$ for any $j \in \Z$. Therefore, it suffices to show that we have a natural quasi-isomorphism $$R\lambda'_*\left(\nu'^*\f\otimes_{\mathcal{O}_X}\gr^j\mathcal{O}\mathbb B_{\operatorname{dR}}\right) \xrightarrow{\simeq} R\lambda_*\left(\nu^*\f\otimes_{\mathcal{O}_X}\gr^j\mathcal{O}\mathbb B_{\operatorname{dR,log}}\right)$$ on $X_{C,\proet},$ for any $j \in \Z$. Since $\f$ is locally free, by applying projection formula to the pullback of $\f$ along $X_{C,\et,\cond} \to X_{\et}$ on both sides, it suffices to show that there exists a natural quasi-isomorphism $$R\lambda'_*\left(\gr^j\mathcal{O}\mathbb B_{\operatorname{dR}}\right) \to R\lambda_*\gr^j\mathcal{O}\mathbb B_{\operatorname{dR,log}}$$ for any $j \in \Z$, which follows from \cite[Lemma 3.7]{lz2017rh} and \cite[Lemma 3.3.2]{dllz2023logrh}.
	\end{proof}

	Now we can prove Theorem \ref{bdrcoh}.
	
	\begin{proof}[Proof of theorem \ref{bdrcoh}]
		The morphism of sites $\lambda:X_{C,\proket} \to X_{C,\et,\cond}$ factors as $$\lambda:X_{C,\proket} \xrightarrow{\mu} X_{C,\proket,\cond} \to X_{C,\et,\cond}.$$             We have $$\rg_{\underline{\proket}}(X_C,\mathbb{M})\simeq \rg_{\proket,\cond}(X_C,R\mu_*\mathbb{M})\simeq \rg_{\et,\cond}(X_C,R\lambda_*\mathbb{M})\simeq \rg_{\et}(X_C,\left(R\lambda_*\mathbb{M}\right)^{\blacktriangledown}).$$ Let $\epsilon:X_{C,\et}\to X_{\et}$ be the base change morphism, then the above property, combined with \cite[Lemma 5.6]{bosco2023padicproetalecohomologydrinfeld} and \cite[Lemma 6.13]{bosco2023padicproetalecohomologydrinfeld}, gives $$\rg_{\underline{\proket}}(X_C,\mathbb{M})\simeq \rg_{\et}(X,R\epsilon_*\left(R\lambda_*\mathbb{M}\right)^{\blacktriangledown}) \simeq \rg_{\et}(X,\underline{\logdr}_X^{\mathcal E} \otimes_K^{\blacksquare}\underline{B_{\operatorname{dR}}}),$$ which conclude the proof of (i).
		
		For the second part, it suffices to show that the natural morphism $$\fil^r(\rg_{\underline{\operatorname{logdR}}}(X,\mathcal{E})\otimes_K^{L_\blacksquare}\underline{B_{\operatorname{dR}}})\to \rg(X,\fil^r(\underline{\logdr}_X^{\mathcal E} \otimes_K^{\blacksquare}\underline{B_{\operatorname{dR}}})) \simeq\rg_{\underline{\proket}}(X_C,\fil^r\mathbb{M}) $$ is a quasi-isomorphism. As both sides satisfy analytic descent, we may assume moreover that $X$ is quasi-compact. Since $X$ is paracompact, the claim then follows from \cite[Theorem 5.20]{bosco2023padicproetalecohomologydrinfeld}, by applying it to each $\fil^i\rg_{\underline{\operatorname{logdR}}}(X,\mathcal{E})\otimes_K^{L_\blacksquare}\fil^j \underline{B_{\operatorname{dR}}}$ for $i+j=r,$ and noting that $\rg(X,-)$ commutes with filtered colimits.
	\end{proof}	
	
	\begin{rem}
		If $X$ is a quasi-compact smooth rigid analytic variety over $K$, we have a $\mathscr{G}_K$-equivariant, compatible with filtrations, natural isomorphisms in $D(\Mod_K^{\cond})$:
		\begin{align*}
			\rg_{\underline{\proket}}(X_C,\mathbb{M})&\simeq  \colim_{j\in \N} \rg_{\underline{\proket}}(X_C,\fil^{-j}\mathbb{M})\\
			&\simeq \colim_{j\in \N}\fil^{-j}(\rg_{\underline{\operatorname{logdR}}}(X,\mathcal{E})\otimes_K^{L_\blacksquare}\underline{B_{\operatorname{dR}}}) \\
			&\simeq \rg_{\underline{\operatorname{logdR}}}(X,\mathcal{E})\otimes_K^{L_\blacksquare}\underline{B_{\operatorname{dR}}}.
		\end{align*}  
		The first isomorphism follows from the fact that the site $X_{\proket}$  is coherent, see \cite[Proposition 5.1.5]{dllz2023logadic}. The second one is the above theorem. The last one follows from the fact that $\otimes_K^{L_\blacksquare}$ commutes with filtered colimits, and filtered colimits are exact in $D(\Mod_K^{\cond})$, as $\Mod_K^{\cond}$ satisfies Grothendieck's axiom (AB5).
	\end{rem}
	
	\begin{rem}
		This theorem allows us to compare the pro-\'etale cohomology and pro-Kummer \'etale cohomology of period sheaves. Suppose that $X$ is a smooth rigid analytic variety over $K$, $D \subset X$ is a strictly normal crossing divisor, and $U=X-D$. Endow $X$ with the log structure coming from $D$. Then the canonical morphism  $$\rg_{\underline{\proet}}(X,\mathbb{B}_{\operatorname{{dR}}}) \xrightarrow{\simeq} \rg_{\underline{\proket}}(U,\mathbb{B}_{\operatorname{{dR}}})$$ is a (non-filtered) quasi-isomorphism in $D(\Mod_K^{\cond})$. This follows from Theorem \ref{logdr}, Theorem \ref{bdrcoh}, and \cite[Theorem 6.5]{bosco2023padicproetalecohomologydrinfeld}.
	\end{rem}
	
	\begin{rem}
		One can define filtered $\mathcal O_{X_{\ket}}$-modules with integrable connections over $X_{\ket}$ in the same way, leading to analogous results. In fact, by \cite[Definition 4.3.6]{dllz2023logadic}, a coherent $\mathcal O_{X_{\ket}}$-module is locally the inverse image of a coherent sheaf on the analytic site of $X$. 
	\end{rem}
	
	\begin{rem}
		If $X$ is a log smooth, proper rigid analytic variety over $K$ and $(\mathcal{E},\nabla,\fil^{\bullet})$ is a filtered $\mathcal O_X$-module with integrable log connection, then by Lemma \ref{finitenessderham} and the finiteness of coherent sheaf cohomology on proper rigid analytic varieties, we have $H^i_{\underline{\operatorname{logdR}}}(X,\mathcal{E})=\underline{H^i_{{\operatorname{logdR}}}(X,\mathcal{E})}$ for all $i\geq 0$. Therefore $$H^i_{\underline{\operatorname{logdR}}}(X,\mathcal{E})\otimes_K^{\blacksquare}\underline{B_{\operatorname{dR}}}=\underline{H^i_{{\operatorname{logdR}}}(X,\mathcal{E})\otimes_KB_{\operatorname{dR}}}.$$
		In particular, this recovers the comparison theorem in \cite[Theorem 3.2.7]{dllz2023logrh}, where they generalize \cite[Corollary 1.8]{scholze2013p}.
	\end{rem}
	
	\section{Logarithmic $B_{\operatorname{dR}}^+$-cohomology for rigid analytic varieties}
	
	In this chapter, we focus on logarithmic rigid analytic varieties over $C$. As outlined in the introduction, establishing a meaningful log de Rham-\'etale comparison theorem requires developing a $B_{\operatorname{dR}}^+$-cohomology theory for such varieties. To achieve this, we adopt the approaches presented in \cite{bms1} and \cite{bosco2023rational}.
	
	This chapter is dedicated to proving Theorem \ref{drcomparison}, which is essential in formulating the semistable conjecture for the geometric case as stated in \cite{xshk}. To maintain focus, we limit our exploration of logarithmic $B_{\operatorname{dR}}^+$-cohomology to the essentials. A more comprehensive study of logarithmic $B_{\operatorname{dR}}^+$-cohomology, along with logarithmic $B$-cohomology, will be presented in a subsequent article.

	\subsection{Preliminary: The $L\eta$ functor}
	
	We review a construction introduced in \cite{berthelot1978crys} and generalized in \cite{bms1}, which is essential for the construction of logarithmic $B_{\operatorname{dR}}^+$-cohomology.
	
	Let $(T,\mathcal O_T)$ be a ringed topos. Denote by $D(\mathcal O_T)$ (resp. $K(\mathcal O_T)$) the derived category (resp. homotopy category) of $\mathcal O_T$-modules. Let $(f)\subset \mathcal O_T$ be an invertible ideal sheaf.
	
	\begin{defn} \cite[Definition 8.6]{berthelot1978crys}
		Let $M^{\bullet}\in K(\mathcal O_T)$ be an $f$-torsion-free complex of $\mathcal O_T$-modules. Define a new $f$-torsion-free complex $\eta_fM^{\bullet}\in K(\mathcal O_T)$ as $$\eta_f(M^{\bullet})^i:=\{x \in f^iM^i|dx \in f^{i+1}M^{i+1}\}$$ for all $i\in\Z$.
	\end{defn}
	
	By \cite[Remark 6.3]{bms1}, the definition of $\eta_f$ depends only on the ideal sheaf $(f)$, and it is independent of the generator of $(f)$. According to \cite[Corollary 6.5]{bms1}, $\eta_f$ factors canonically over a (not triangulated!) functor $$L\eta_f:D(\mathcal O_T)\to D(\mathcal O_T)$$ on derived categories.
	
	We now define a filtration on $L\eta_fM$. We refer the reader to \cite[Chapter 5]{bms2} for notation and basic properties of filtered derived $\infty$-category of $\mathcal O_T$-modules. Recall that the filtered derived $\infty$-category of $\mathcal O_T$-modules is defined to be $$DF(\mathcal O_T):=\operatorname{Fun}(\Z^{\operatorname{op}},D(\mathcal O_T)).$$ For $F\in DF(\mathcal O_T)$, we denote by $$\gr^i(F):=F(i)/F(i+1)$$ the $i$-th graded piece of $F$. 
	
	\begin{defn} \label{decalagefiltration} \label{filtration}
		For $M\in D(\mathcal O_T)$, we define $\fil^{\star}L\eta_fM$ a filtration on $L\eta_fM$ as follows: choose a representation $M^{\bullet}\in K(\mathcal O_T)$ of $M$ such that each $M^i$ is $f$-torsion-free. Then the $i$-th level of $\fil^{\star}L\eta_fM$ is given by $$\fil^{i}L\eta_fM:=f^iM^{\bullet}\cap \eta_fM^{\bullet}.$$ The fact that this filtration is well-defined follows from Proposition \ref{welldefine}.
	\end{defn}
	
	To understand the filtration on $L\eta_fM$, we recall the following definition from \cite{bms2}, which elaborates of \cite[Appendix A]{beilinson1987perverse}.	
	
	\begin{defn} \cite[Definition 5.3]{bms2}
		Let $ DF^{\leq 0}(\mathcal O_T) \subset DF(\mathcal O_T)$ be the full subcategory consisting of objects $ F $ such that $ \operatorname{gr}^i(F) \in D^{\leq i}(\mathcal O_T) $ for all $ i $. Dually, $ DF^{\geq 0}(\mathcal O_T) \subset DF(\mathcal O_T)$ is the full subcategory consisting of objects $ F $ such that $ F(i) \in D^{\geq i}(\mathcal O_T) $ for all $ i $. We call the pair $(DF^{\leq 0}(\mathcal O_T), DF^{\geq 0}(\mathcal O_T)) $ the Beilinson $ t $-structure on $ DF(\mathcal O_T)$.
	\end{defn}
	
	This definition is justified by Theorem \ref{beilinsont} below.
	
	\begin{thm} \cite[Theorem 5.4]{bms2} \label{beilinsont}		
		(1) The Beilinson $t$-structure $(DF^{\leq 0}(\mathcal O_T), DF^{\geq 0}(\mathcal O_T))$ is a $t$-structure on $DF(\mathcal O_T)$.
		
		(2) Let $\tau^{\leq 0}_B$ denote the connective cover functor associated with the $t$-structure from (1). Then, there exists a natural isomorphism $$\operatorname{gr}^i \circ \tau^{\leq 0}_B(-) \simeq \tau^{\leq i} \circ \operatorname{gr}^i(-).$$ 
		
		(3) The heart of the Beilinson $t $-structure, defined as $$DF(\mathcal O_T)^\heartsuit := DF^{\leq 0}(\mathcal O_T) \cap DF^{\geq 0}(\mathcal O_T),$$ is equivalent to the abelian category $ \operatorname{Ch}(\mathcal O_T)$ of chain complexes of $\mathcal O_T$-modules. This equivalence can be described as follows: for $ F \in DF(\mathcal O_T)$, its 0-th cohomology $ H^0_B(F)$ in the Beilinson $t$-structure corresponds to the chain complex $(H^\bullet(\gr^\bullet(F)), d)$, where \( d \) is the boundary map induced by the exact triangle $$\gr^{i+1}(F) = {F(i+1)}/{F(i+2)} \to {F(i)}/{F(i+2)} \to  {F(i)}/{F(i+1)}=\gr^i(F).$$
	\end{thm}

	We cite the following proposition of \cite{bms2}, which describes the structure of the filtration of the d\'ecalage functor.
	
	\begin{prop} \cite[Proposition 5.8]{bms2} \label{welldefine}
		Fix $M\in D(\mathcal O_T).$ Let $f^{\star}\otimes M$ be the $f$-adic filtration on $M$, i.e., the $i$-th level of the filtration of $M$ is $f^i\otimes M$ with obvious maps. Then $L\eta_fM$ is filtered isomorphic with $\tau_B^{\leq 0}(f^{\star}\otimes M)$ in $DF(\mathcal O_T).$
	\end{prop}
	
	We list some results of $L\eta_f$, which will be used in the rest of this chapter.
	
	\begin{prop} \label{propleta}
		(1)\cite[Lemma 6.10]{bms1} If $M \in D^{\geq 0}(\mathcal O_T)$ such that $H^0(M)[f]=0.$ Then there is a canonical map $\fil^{\star}L\eta_fM\to f^{\star}\otimes M$ in $DF(\mathcal O_T)$.
		
		(2)\cite[Proposition 6.12]{bms1} For any $M \in D(\mathcal O_T)$, we construct a complex $H^{\bullet}(M/(f))$ with terms $$H^i(M/(f)):=H^i(M/^{ L}(f))\otimes_{\mathcal{O}_T}(f^i),$$ and differential induced by the boundary map corresponding to the short exact sequence $$0\to (f)/(f^2)\to \mathcal{O}_T/(f^2) \to \mathcal{O}_T/(f)\to 0.$$ Then there is a natural quasi-isomorphism $$L\eta_fM\otimes^L_{\mathcal{O}_T}\mathcal{O}_T/(f)\xrightarrow{\simeq}H^{\bullet}(M/(f)).$$
		
		(3)\cite[Lemma 6.14]{bms1} Let $u:(T,\mathcal O_T)\to (T',\mathcal O_{T'})$ be a flat map of ringed topoi. Let $(f)\subset \mathcal O_T$ be an invertible ideal sheaf with $(f'):=u^*(f)\subset \mathcal O_{T'}$, which is still invertible. Then the diagram 
		$$\begin{tikzcd}
			D(\mathcal O_T) \arrow[r,"u^*"] \arrow[d,"L\eta_f"] & D(\mathcal O_{T'}) \arrow[d,"L\eta_{f'}"] \\
			D(\mathcal O_T) \arrow[r,"u^*"] & D(\mathcal O_{T'})
		\end{tikzcd}$$
		commutes, i.e., there is a natural quasi-isomorphism $L\eta_{f'}u^*M\to u^*L\eta_fM$ in $D(\mathcal O_{T'})$ for all $M \in D(\mathcal O_T)$.
	\end{prop}
	
	\begin{proof}
		We only need to check the compatibility for filtration of (1), which follows from the construction in \cite[Lemma 6.10]{bms1}: choosing a representation $M^{\bullet}$ of $M$, once we get a map $L\eta_fM\to f^{\star}\otimes M$, we can restrict it to $\fil^{\star}L\eta_fM=f^{\star}M^{\bullet}\cap \eta_fM^{\bullet}\to f^{\star}\otimes M.$
	\end{proof}

	\subsection{Logarithmic $B_{\operatorname{dR}}^+$-cohomology}
	
	Similar to \cite{bosco2023rational}, we can define the logarithmic $B_{\operatorname{dR}}^+$-cohomology by using the $L\eta$ functor.
	
	\begin{defn}
		Let $X$ be a log smooth rigid analytic variety over $C$. We denote by $\lambda:X_{\proket}\to X_{\et,\operatorname{cond}}$ the natural morphism of sites. We define the logarithmic $B_{\operatorname{dR}}^+$-cohomology of $X$ as the complex in $D(\Mod_{B_{\operatorname{dR}}^+}^{\operatorname{solid}})$: $$\rg_{B_{\operatorname{dR}}^+}(X):=\rg_{\et,\mathrm{cond}}(X,L\eta_tR\lambda_*\mathbb B_{\operatorname{dR}}^+).$$ We endow $\rg_{B_{\operatorname{dR}}^+}(X)$ with the filtration given by Definition \ref{decalagefiltration}.
	\end{defn}
	
	\begin{rem}
		Similar to \cite{guo2021crystalline}, by introducing the logarithmic infinitesimal sites, one is able to define $\rg_{\underline{\operatorname{logdR}}}(X/B_{\operatorname{dR}}^+)$, and there should be a canonical filtered quasi-isomorphism: $$\rg_{B_{\operatorname{dR}}^+}(X) \simeq \rg_{\underline{\operatorname{logdR}}}(X/B_{\operatorname{dR}}^+).$$
	\end{rem}
	
	\begin{prop} \label{bdrboundness}
		Let $X$ be a log smooth rigid analytic variety of dimension $d$ over $C$. Then $R^i\lambda_*\mathbb B_{\operatorname{dR}}^+=0$ for all $i>d.$
	\end{prop}
	
	\begin{proof} 
		This is an analogy of \cite[Proposition 2.40]{bosco2023rational}. Since the problem is local on $X_{\et,\cond}$ (as both sides satisfy analytic descent), following the notations in Section 3, we may assume $X=\Spa(R,R^+)$ is an affinoid log rigid space over $C$ with a strictly \'etale morphism $X \to \mathbb E_C$, which can be written as a composite of rational embeddings and finite \'etale maps. By pulling back $\widehat{\widetilde{\mathbb E}}_C \to \mathbb E_C$ to $X$ we get $$\widetilde{X}:=X\times_{\mathbb E_C}\widehat{\widetilde{\mathbb E}}_C=\Spa(R_{\infty},R^+_{\infty}),$$a Galois pro-finite Kummer \'etale cover with Galois group $\Gamma$. The Cartan–Leray spectral sequence gives an quasi-isomorphism $$\rg_{\operatorname{cond}}(\Gamma,H^0(\widetilde{X},\mathbb{B^+_{\operatorname{dR}}}))\xrightarrow{\simeq}\rg_{\underline{\proket}}(X,\mathbb{B^+_{\operatorname{dR}}}),$$ which concludes the proof since $\Gamma \simeq \widehat{\Z}(1)^d$ has $p$-cohomological dimension $d$ by \cite[I.3.3, Proposition 14]{serre2001galois}: since $\Z_p^d(1)$ is the $p$-Sylow subgroup of $\Gamma$ (see \cite[1.4]{serre2001galois} for the definition), the claim then follows from \cite[B.3]{bosco2023padicproetalecohomologydrinfeld}.
	\end{proof}
	
	This proposition allows us to deduce the following two propositions.
	
	\begin{prop} \label{bdrcompleteness}
		The natural map $$L\eta_tR\lambda_*\mathbb B^+_{\operatorname{dR}}\to R\varprojlim_{m}L\eta_tR\lambda_* (\mathbb B^+_{\operatorname{dR}}/t^m)$$ is a quasi-isomorphism compatible with filtrations induced by $L\eta_t$.
	\end{prop}
	
	\begin{proof}
		The proof is similar to \cite[Lemma 4.3]{bras2018overconvergentrelativerhamcohomology} or \cite[Lemma 2.41]{bosco2023rational}.
	\end{proof}
	
	\begin{prop} \label{bdraffinoid}
		If $X$ is an affinoid log smooth rigid analytic variety over $C$, the natural morphism $$L\eta_t\rg_{\underline{\proket}}(X,\mathbb B_{\operatorname{dR}}^+)\to \rg_{\et,\mathrm{cond}}(X,L\eta_tR\lambda_*\mathbb B_{\operatorname{dR}}^+)=\rg_{B_{\operatorname{dR}}^+}(X)$$ in $D(\Mod_{B_{\operatorname{dR}}^+}^{\operatorname{solid}})$ is a filtered quasi-isomorphism, where filtrations on both sides are given by the filtration of $L\eta_t$ defined in Definition \ref{filtration}.
	\end{prop}
	
	\begin{proof}
		The proof is similar to \cite[Proposition 3.11]{bras2018overconvergentrelativerhamcohomology} or \cite[Proposition 2.42]{bosco2023rational}.
	\end{proof}
	
	\begin{thm} \label{bdrK}
		Let $X$ be a log smooth rigid analytic variety over $K$ then we have a natural isomorphism in $D(\Mod_K^{\operatorname{solid}}):$ $$\rg_{B_{\operatorname{dR}}^+}(X_C)\xrightarrow{\simeq}\rg_{\underline{\operatorname{logdR}}}(X)\otimes_{K}^{L_{\blacksquare}}\underline{B_{\operatorname{dR}}^+},$$ compatible with filtration and $G_K$-action.
	\end{thm}
	
	\begin{proof}
		Since the problem is local (as both sides satisfy analytic descent), we may assume that $X$ is affinoid. In this case,  recall that $\lambda: X_{\proket}\to X_{C,\et,\cond},$ we have 
		\begin{align*}
			\rg_{B_{\operatorname{dR}}^+}(X_C)&\simeq  L\eta_t\rg_{\underline{\proket}}(X_C,\mathbb B_{\operatorname{dR}}^+)\\
			&\simeq L\eta_t\rg_{\et,\cond}(X_C,(R\lambda_*\mathbb B_{\operatorname{dR}}^+)^{\blacktriangledown}) \\
			&\simeq L\eta_t \rg_{\et,\cond}(X,\fil^0(\underline{\logdr}_{X} \otimes_K^{\blacksquare}\underline{B_{\operatorname{dR}}})) \\
			&\simeq \left(\underline{\mathcal{O}_{X}(X)}\xrightarrow{\nabla} \underline{\Omega_{X}^{1,\log}(X)}\xrightarrow{\nabla} \underline{\Omega_{X}^{2,\log}(X_)}\xrightarrow{\nabla} \cdots \right) \otimes_{K}^{L_\blacksquare}\underline{B_{\operatorname{dR}}^+}\\
			&\simeq  \rg_{\underline{\operatorname{logdR}}}(X)\otimes_{K}^{L_\blacksquare}\underline{B_{\operatorname{dR}}^+},
		\end{align*}  
		where the first quasi-isomorphism follows from Proposition \ref{bdraffinoid}, the second quasi-isomorphism follows from the same argument as in the proof of Theorem \ref{bdrcoh}, the third quasi-isomorphism follows from Proposition \ref{bdrdirectimage}, the forth quasi-isomorphism follows from \cite[Lemma 6.13]{bosco2023padicproetalecohomologydrinfeld} and \cite[Theorem 5.20]{bosco2023padicproetalecohomologydrinfeld}, and the last quasi-isomorphism follows from \cite[Lemma 5.6]{bosco2023padicproetalecohomologydrinfeld}. This completes the proof.
	\end{proof}
	
	When $X$ is defined over $C$, our construction of the $B_{\operatorname{dR}}^+$-cohomology theory indeed provides a deformation of the logarithmic de Rham cohomology. Moreover, we can describe its filtration in a manner analogous to \cite[Proposition 3.13]{CN4.3} in the rigid analytic setting, as follows.
	
	\begin{thm}  \label{bdrC}
		Let $X$ be a log smooth rigid analytic variety over $C$.
		
		(1)	We have a natural isomorphism in $D(\Mod_C^{\operatorname{solid}}):$ $$\theta:\rg_{B_{\operatorname{dR}}^+}(X)\otimes_{B_{\operatorname{dR}}^+}^{L_\blacksquare}C\xrightarrow{\simeq}\rg_{\underline{\operatorname{logdR}}}(X).$$
		
		(2) More generally, for $r \geq 0,$ we have a natural distinguished triangle in $D(\Mod_C^{\operatorname{solid}}):$ $$\fil^{r-1}\rg_{B_{\operatorname{dR}}^+}(X) \xrightarrow{t}\fil^{r}\rg_{B_{\operatorname{dR}}^+}(X) \xrightarrow{\theta}\fil^{r}\rg_{\logdr}(X).$$
		
		(3) For $r \geq 0,$ we have a natural distinguished triangle in $D(\Mod_C^{\operatorname{solid}}):$ $$\fil^{r+1}\rg_{B_{\operatorname{dR}}^+}(X) \xrightarrow{}\fil^{r}\rg_{B_{\operatorname{dR}}^+}(X)\xrightarrow{}\rg(X,\tau^{\leq r}R\lambda_*\widehat{\mathcal O}_{X_{\proket}}(r)).$$
	\end{thm}

	\begin{proof}
		We have $$\rg_{B_{\operatorname{dR}}^+}(X)\otimes_{B_{\operatorname{dR}}^+}^{L_\blacksquare}C \simeq \rg_{\et,\operatorname{cond}}(X,L\eta_tR\lambda_*\mathbb B_{\operatorname{dR}}^+\otimes_{B_{\operatorname{dR}}^+}^{L_\blacksquare}(B_{\operatorname{dR}}^+/t)).$$ We are reduced to showing $$\left(L\eta_tR\lambda_*\mathbb B_{\operatorname{dR}}^+\otimes_{B_{\operatorname{dR}}^+}^{L_\blacksquare}(B_{\operatorname{dR}}^+/t)\right)^{\blacktriangledown} \simeq \underline{\Omega_X^{\bullet,\log}}.$$ By Proposition \ref{propleta} (2), we have
		\begin{align*}
			\left(L\eta_tR\lambda_*\mathbb B_{\operatorname{dR}}^+\otimes_{B_{\operatorname{dR}}^+}^{L_\blacksquare}(B_{\operatorname{dR}}^+/t)\right)^{\blacktriangledown}&\simeq  H^{\bullet}((R\lambda_*\mathbb B_{\operatorname{dR}}^+)/t)^{\blacktriangledown}
			\simeq H^{\bullet}(R\lambda_*(\mathbb B_{\operatorname{dR}}^+/t))^{\blacktriangledown}
			\simeq H^{\bullet}(R\lambda_*\widehat{\mathcal O}_{X_{\proket}})^{\blacktriangledown}  \\
			&\simeq \left(R^{\bullet}\lambda_*\widehat{O}_{X_{\proket}}\right)^{\blacktriangledown}\otimes_{B_{\operatorname{dR}}^+}\fil^{\bullet}B_{\operatorname{dR}}^+\\
			&\simeq  \underline{\Omega_X^{\bullet,\log}(-\bullet)}\otimes_CC(\bullet) \\
			&\simeq \underline{\Omega_X^{\bullet,\log}},
		\end{align*} 
		where the second-to-last isomorphism follows from the proposition below: to see that the Bockstein-type differential coincides with the usual log-differential on $\Omega_X^{\bullet,\log}$ (this is a much easier version of \cite[Theorem 8.3]{bms1}, see also \cite[Proposition 7.9]{bhatt2015specializing}), by Proposition \ref{pushforward} we may reduce to checking that the Bockstein-type map $$\mathcal O_X \simeq R^0\lambda_*\widehat{\mathcal O}_{X_{\proket}} \to R^1\lambda_*\widehat{\mathcal O}_{X_{\proket}}(1) \simeq \Omega_X^{1,\log}$$ is the usual log-differential. We may work locally, translating the elements in $R^1\lambda_*\widehat{\mathcal O}_{X_{\proket}}(1)$ to cocycle conditions in group cohomology, where the same calculation of the proof of Proposition \ref{htsplitting} applies. This completes the proof of (1). Then (2) follows from \cite[Proposition 3.4]{wu2024decalage}, and (3) follows from \cite[Lemma 3.3]{wu2024decalage}.
	\end{proof}
	
	\begin{rem}
		We will see that when $X$ descends to a log smooth $X_0$ over a discretely valued field $K$ such that $X=X_0 \times_K C$, we have a natural quasi-isomorphism $$\tau^{\leq r}R\lambda_*\widehat{\mathcal O}_{X_{\proket}}(r)\simeq \bigoplus_{i\leq r}\Omega_X^{i,\log}(r-i)[-i],$$ and in general, a choice of a log smooth $B_{\operatorname{dR}}^+/t^2$-lift $\mathbb X$ of $X$ via the map $B_{\operatorname{dR}}^+/t^2 \to C$ gives such a quasi-isomorphism.
	\end{rem}
	
	\begin{prop} \label{pushforward}
		Let $X$ be a log smooth rigid analytic variety over $C$, then we have a natural isomorphism $$\left(R^i\lambda_*\widehat{\mathcal O}_{X_{\proket}}\right)^{\blacktriangledown}\simeq \underline{\Omega_X^{i,\log}(-i)}$$ for all $i \geq 0.$
	\end{prop}
	
	\begin{proof}
		We can use the same argument as in the proof of \cite[Proposition 3.23]{scholze2013perfectoidspacessurvey}: we first prove that $\mathcal E:=R^1\lambda_*\widehat{\mathcal O}_{X_{\proket}}$ is a locally free $\mathcal O_{X_{\et}}$-module of rank $\dim X$, such that $$\bigwedge^j \mathcal{E}\simeq R^j\lambda_*\widehat{\mathcal O}_{X_{\proket}}.$$ This can be checked locally, by following a direct computation using the Koszul complex of $\Gamma$ (as in the proof of \cite[Lemma 5.5]{scholze2013p}, which also requires the use of \cite[Lemma 6.1.9]{dllz2023logadic}). Then, the following lemma, which is a log version of \cite[Lemma 3.24]{scholze2013perfectoidspacessurvey}, concludes the proof.
	\end{proof}
	
	\begin{lem} \label{constructpushforward}
		Consider the exact sequence (for example, see \cite[Proposition 2.3]{kato1999logcoh}) $$0\to \Z_p(1)\to \varprojlim_{\times p}\nu^{-1}\mathcal M_X^{gp}\to \nu^{-1}\mathcal M_X^{gp}\to 0$$ on $X_{\proket}.$ Here $\nu$ is the natural morphism of sites $\nu:X_{\proket}\to X_{\et}$. It induces a boundary map ${\mathcal M_X^{gp}}=\nu_*\nu^{-1}{\mathcal M_X^{gp}} \to  R^1\lambda_*\Z_p(1).$ Then there is a unique $\underline{\mathcal O_{X}}$-linear map $\underline{\Omega_X^{\log}}\to R^1\lambda_*\widehat{\mathcal O}_{X_{\proket}}(1)^{\blacktriangledown}$ such that the diagram
		$$\begin{tikzcd}
			\mathcal M_X^{gp} \arrow[r] \arrow[d,"\delta"] & R^1\nu_*\Z_p(1) \arrow[d] \\
			\underline{\Omega_X^{\log}}(*) \arrow[r] & R^1\lambda_*\widehat{\mathcal O}_{X_{\proket}}(1)^{\blacktriangledown}(*)
		\end{tikzcd}$$ commutes, where $\delta:\mathcal M \to \Omega^{1,\log}$ is the derivation. This map is an isomorphism.
	\end{lem}
	
	\begin{proof}
		The proof of \cite[Lemma 3.24]{scholze2013perfectoidspacessurvey} goes through. We record the proof for completeness. 
		
		The assertion is local, so we may assume $X=\Spa(R,R^+)$ ($R^+=R^{\circ}$) is an affinoid log rigid space over $C$ with a strictly \'etale morphism $X \to \mathbb E=\Spa(C\langle P \rangle,\mathcal O_C\langle P \rangle)$ for some toric monoid $P$, which can be written as a composition of rational localizations and finite \'etale maps. Choose a $\Z$-basis $\{a_1,...,a_d\}$ of $P^{\operatorname{gp}}$.
		
		It suffices to prove the existence. For this, we have to check that the diagram
		\[
		\begin{tikzcd}
			H^0(X_{\text{prok\'et}}, \mathcal M_X^{gp}) \ar[r] \ar[d] & H^1(X_{\text{prok\'et}}, \widehat{\mathbb{Z}}_p(1)) \ar[d] \\
			H^0(X_{\text{prok\'et}}, \Omega_X^{\log}) \ar[r] & H^1(X_{\text{prok\'et}}, \widehat{\mathcal{O}}_X(1))
		\end{tikzcd}
		\]	
		commutes, where the lower map is defined to be the one determined by the $\delta(a_i)$. 
		First, by the local computation, $H^1(X_{\text{pro\'et}}, \widehat{\mathcal{O}}_X(1))$ is a finite free $R$-module. 
		In particular, it carries a natural topology. Similar to the proof of \cite[Lemma 3.24]{scholze2013perfectoidspacessurvey}, via both maps, the image of $R^{+\times}$ in $H^0(X_{\text{prok\'et}}, \mathcal M_X^{gp})$ is bounded in $H^1(X_{\text{prok\'et}}, \widehat{\mathcal{O}}_X(1))$.
		
		As explained in the proof of \cite[Lemma 3.24]{scholze2013perfectoidspacessurvey}, we can find formally smooth $\mathbb{Q}_p$-algebras $B_i$ topologically of finite type with $C = \widehat{\varinjlim B_i}$. 
		We may also assume that there are \'etale maps $Y_i = \operatorname{Spa}(B_i, B_i^+) \to \mathbb{T}^{n_i}$ for some $n_i$. 
		For $i$ large enough, the underlying space $X$ comes as base extension from $X_i \to \Spa(B_i\langle P \rangle, B_i^+\langle P \rangle) \to Y_i$. We endow $X_i$ with the log structure associated to the natural pre-log structure $P \to B_i\langle P \rangle.$ Then $X$ comes as base extension from $X_i  \to Y_i$ as a fs log adic spaces.
		Let $X_i = \operatorname{Spa}(R_i, R_i^+)$, where now $R_i$ is topologically of finite type over $\mathbb{Q}_p$. 
		Moreover, given any $f \in \mathcal M_X^{gp}(X)$ and $m \geq 1$, as $\mathcal M_X^{gp}(X)$ is generated by $R^{\times}$ and $P^{gp}$, one can find an $i$ and some $f_i \in \mathcal M_{X_i}^{gp}(X_i)$ such that $f-{f_i} \in \mathcal M_X(X),$ and that $\alpha(f-f_{i})\in 1 + p^m R^+$, where we recall that $\alpha$ is the structure morphism $\alpha:\mathcal M_X \to \mathcal{O}_X$. 
		By the boundedness statement on the image of $R^{+\times}$, and because elements of $1 + p^m R^+$ are $p^{m-2}$-th powers of elements of $1 + p^2 R^+ \subset R^{+\times}$ (say $m \geq 2$), it suffices to prove the statement for $f = f_i$.

		There is an exact commutative diagram
		\[
		\begin{tikzcd}
			0 \ar[r] & {\mathbb{Z}}_p(1) \ar[r] \ar[d] & \varprojlim_{\times p}\nu_i^{-1}\mathcal M_{X_i}^{gp} \ar[r] \ar[d] & \nu_i^{-1}\mathcal M_{X_i}^{gp} \ar[r] \ar[d,"\delta"] & 0 \\
			0 \ar[r] & \widehat{\mathcal{O}}_{X_i}(1) \ar[r] & \mathcal{F}_i \ar[r] & \widehat{\mathcal{O}}_{X_i} \otimes_{\mathcal{O}_{X_i}} \Omega_{X_i}^{\log} \ar[r] & 0
		\end{tikzcd}
		\]	
		on $(X_i)_{\text{prok\'et}}$. Here, the bottom row comes from the Faltings's extension \cite[Corollary 2.4.5]{dllz2023logrh}, and $\nu_i$ is the natural morphism of sites $\nu_i:X_{i,\proket}\to X_{i,\et}$. An element of $\varprojlim_{\times p}\nu_i^{-1}\mathcal M_{X_i}^{gp}$ can be represented by $U_1-U_2$ where $U_1,U_2 \in {\mathcal{M}}_{X_i}^{\flat}$, which gives $V_1,V_2 \in \mathcal{M}_{X_i}$; then $(V_1/[U_1])-(V_2/[U_2]) \in \mathcal{O}\mathbb B^+_{\operatorname{dR,log},X_i}$ defines an element of $\mathcal{F}_i$; this gives the middle vertical map. Applying $g_i^{-1}$ (where $g_i : X_{\text{prok\'et}} \to (X_i)_{\text{prok\'et}}$) to the diagram and tensoring the lower sequence with $\widehat{\mathcal{O}}_X$ over $g_i^{-1} \widehat{\mathcal{O}}_{X_i}$ gives an exact commutative diagram 
		\[
		\begin{tikzcd}
			0 \ar[r] & g_i^{-1}{\mathbb{Z}}_p(1) \ar[r] \ar[d] & g_i^{-1}\varprojlim_{\times p}\nu_i^{-1}\mathcal M_{X_i}^{gp} \ar[r] \ar[d] & g_i^{-1}\nu_i^{-1}\mathcal M_{X_i}^{gp} \ar[r] \ar[d,"\delta"] & 0 \\
			0 \ar[r] & \widehat{\mathcal{O}}_{X}(1) \ar[r] & \mathcal{F}'_i \ar[r] & \widehat{\mathcal{O}}_{X} \otimes_{\mathcal{O}_{X}} (\Omega^{\log}_X)' \ar[r] & 0,
		\end{tikzcd}
		\]	 
		where $\f_i'$ is some sheaf of $\widehat{\mathcal{O}}_X$-module, and $(\Omega^{\log}_X)':=g_i^{-1}\Omega_{X_i}^{\log}\otimes_{g_i^{-1}{\mathcal{O}}_{X_i}}{\mathcal{O}}_{X}$ is a free ${\mathcal{O}}_{X}$-module with a projection $(\Omega^{\log}_X)'\to \Omega^{\log}_X$. The kernel comes from $\Omega_{Y_i}$ and is generated by the image of $\mathcal{O}_{Y_i}^{\times}$ via $d\log$.
		As explained in the proof of \cite[Lemma 3.24]{scholze2013perfectoidspacessurvey}, by comparing the boundary maps, one finds that the map
		$$s: H^0(X_{\text{prok\'et}}, \widehat{\mathcal{O}}_X \otimes_{\mathcal{O}_X} (\Omega^{\log}_X)') \to H^1(X_{\text{prok\'et}}, \widehat{\mathcal{O}}_X(1))$$
		defined by the $\widehat{\mathcal{O}}_X$-linear extension $\mathcal{F}'_i$ is $\widehat{\mathcal{O}}_X$-linear, and maps $\delta(a_i)$ to the correct elements. Moreover, $s$
		is the composite of the projection $H^0(X_{\text{prok\'et}}, \widehat{\mathcal{O}}_X \otimes_{\mathcal{O}_X} (\Omega^{\log}_X)') \to H^0(X_{\text{pro\'et}}, \widehat{\mathcal{O}}_X \otimes_{\mathcal{O}_X} \Omega^{\log}_X)$ with the map
		$H^0(X_{\text{prok\'et}}, \widehat{\mathcal{O}}_X \otimes_{\mathcal{O}_X} \Omega^{\log}_X) \to H^1(X_{\text{prok\'et}}, \widehat{\mathcal{O}}_X(1))$ pinned down by the $\delta(a_i)$. Now, use the commutativity of the diagrams for $f_i \in H^0(X_{\text{prok\'et}}, g_i^{-1}\nu_i^{-1} \mathcal{M}_{X_i}^{gp})$ to conclude.
	\end{proof}
	
	Similar to the log de Rham cohomology, we can compare the logarithmic $B_{\operatorname{dR}}^+$-cohomology with the $B_{\operatorname{dR}}^+$-cohomology of the trivial locus when the log structure comes from a strictly normal crossing divisor.
	
	\begin{prop}
		Let $X$ be a log smooth rigid analytic variety over $C$ with log structure given by a strictly normal crossing divisor $D \subset X$. Denote by $U:=X-D$. Then we have a natural non-filtered isomorphism in $D(\Mod_{B_{\operatorname{dR}}^+}^{\operatorname{solid}}):$ $$\rg_{B_{\operatorname{dR}}^+}(X)\simeq \rg_{\underline{\operatorname{dR}}}(U/B_{\operatorname{dR}}^+).$$
	\end{prop}
	
	\begin{proof}
		Since both sides satisfy analytic descent, the problem is local. The lemma below shows that we can assume $X$ is the base change to $C$ of $\Sp(A\langle x_1,...,x_d\rangle)$ with log structure given by $x_1\cdots x_d=0,$ where $\Sp(A)$ is a smooth rigid analytic variety over $K$ for some discretely valued subfield $K$ of $C$. Then the claim follows from Theorem \ref{bdrK}, \cite[Theorem 5.1]{bosco2023rational} and Theorem \ref{logdr}.
	\end{proof}
	
	\begin{lem}
		If $X$ is a log smooth rigid analytic variety over $C$ of dimension $d$ with log structure given by a strictly normal crossing divisor $D \subset X$, then \'etale locally $X$ has a basis of rigid analytic varieties consisting of objects of the form $\Sp(A_C\langle x_1,...,x_d\rangle)$ with log structure given by the normal crossing divisor $x_1\cdots x_d=0,$ where $\Sp(A)$ is a smooth rigid analytic variety over $K$ for some discretely valued subfield $K$ of $C$.
	\end{lem}
	
	\begin{proof}
		By the existence of tubular neighborhoods, \cite[Theorem 1.18]{kiehl1967derham}, locally $X$ can be written as $V:=\Sp(R\left\langle x_1,...,x_d\right\rangle)$ for some smooth affinoid algebra over $C$, and $D$ is defined by the equation $x_1\cdots x_d=0.$ Then the claim follows from \cite[Theorem 7 and Remark 2]{elkik1973solutions}, which asserts that $R$ is the base change to $C$ of a smooth affinoid algebra defined over a discretely valued field $K$.
	\end{proof}
	
	\subsection{Comparison with the pro-Kummer \'etale cohomology}
	
	The following comparison theorem can be easily deduced from our construction for logarithmic $B_{\operatorname{dR}}^+$-cohomology.
	
	\begin{thm} \label{bdrcomparison}
		Let $X$ be a log smooth proper rigid analytic variety over $C$. Then we have a canonical filtered isomorphism in $\operatorname{CondAb}:$ $$H^i_{\underline{\ket}}(X,\Q_p) \otimes_{\Q_p}B_{{\operatorname{dR}}} \simeq H^i_{B_{\operatorname{dR}}^+}(X) \otimes_{B_{\operatorname{dR}}^+}B_{\operatorname{dR}}.$$ Here, the filtration on $H^i_{B_{\operatorname{dR}}^+}(X)$ is defined by $$\fil^{\star}H^i_{B_{\operatorname{dR}}^+}(X):=\operatorname{Im}(H^i(\fil^{\star}\rg_{B_{\operatorname{dR}}^+}(X))\to H^i_{B_{\operatorname{dR}}^+}(X)).$$ Moreover, when $X$ descends to a log smooth $X_0$ over a discretely valued field $K$ such that $X=X_0 \times_K C$, this isomorphism agrees with the comparison isomorphism $$H^i_{\underline{\ket}}(X,\Q_p) \otimes_{\Q_p}B_{\operatorname{dR}} \simeq H^i_{\underline{\operatorname{logdR}}}(X_0) \otimes_{K}B_{\operatorname{dR}},$$ under the canonical identification in Theorem \ref{bdrK}.
	\end{thm}
	
	\begin{proof}
		We will construct a morphism $$\rg_{B_{\operatorname{dR}}^+}(X)\to \rg_{\underline{\proket}}(X,\mathbb B_{\operatorname{dR}}^+)\simeq \rg_{\underline{\ket}}(X,\Z_p)\otimes_{\Z_p} B_{\operatorname{dR}}^+,$$ where the last isomorphism follows from \cite[Lemma 3.6.1]{dllz2023logrh} (whose proof also works over $C$, by the same proof as presented in \cite[Theorem 8.4]{scholze2013p}). We will then prove that such a morphism is a quasi-isomorphism after inverting $t$.
		
		By Proposition \ref{propleta} (1), we have a natural morphism $$\rg_{B_{\operatorname{dR}}^+}(X)= \rg_{\et,\operatorname{cond}}(X,L\eta_tR\lambda_*\mathbb B_{\operatorname{dR}}^+)\to \rg_{\et,\operatorname{cond}}(X,R\lambda_*\mathbb B_{\operatorname{dR}}^+)= \rg_{\underline{\proket}}(X,\mathbb B_{\operatorname{dR}}^+),$$ which gives the desired morphism.
		
		To show that such a morphism is a quasi-isomorphism after inverting $t$, as both sides satisfy analytic descent, it suffices to prove it locally so we can assume that $X$ is affinoid. By Proposition \ref{propleta} (3), considering the base change $B_{\operatorname{dR}}^+ \to B_{\operatorname{dR}},$ picking $\mathcal O_{T'}:=B^+_{\operatorname{dR}},\mathcal O_{T}:=B_{\operatorname{dR}}$ and $M:=\rg_{\underline{\proket}}(X,\mathbb B_{\operatorname{dR}}^+)$, and because $L\eta_t$ sends a $t$-invertible object to itself, we have a quasi-isomorphism $$\rg_{B_{\operatorname{dR}}^+}(X)\left[\dfrac{1}{t}\right]\simeq L\eta_t\left( \rg_{\underline{\proket}}(X,\mathbb B_{\operatorname{dR}}^+)\left[\dfrac{1}{t}\right]\right) \simeq \rg_{\underline{\proket}}(X,\mathbb B_{\operatorname{dR}}),$$ which is compatible with filtration: for any integer $r$, we have $$\fil^r\left(L\eta_tM\left[\dfrac{1}{t}\right]\right)=\sum_{i+j=r}t^j\left(t^iM \cap L\eta_tM\right)\subseteq t^rM=t^r\rg_{\underline{\proket}}(X,\mathbb B^+_{\operatorname{dR}}).$$ This completes the proof. 
		
		When $X$ descends to $X_0$ over a discretely valued field $K$, the compatibility stated in the theorem is clear due to Theorem \ref{bdrK}.
	\end{proof}
	
	\subsection{Degeneration of (log) Hodge--Tate spectral sequence}
	
	This section is devoted to proving the following theorem, as promised in the introduction. For convenience, we will only work with usual cohomology groups in this section. However, clearly all results can be extended to condensed cohomology group without any difficulty.
	
	Recall that we denote by $\nu$ the natural morphism of sites $\nu:X_{\proket}\to X_{\et}$.
	
	\begin{thm} \label{degeneration}
		Let $X$ be a proper log smooth rigid analytic variety over $C$ of dimension $d$, then
		
		(i) the Hodge--log de Rham spectral sequence $$E^{ij}_1:=H^j(X,\Omega^{\log,i}_{X})\Rightarrow H^{i+j}_{\operatorname{logdR}}(X)$$ degenerates at $E_1$.
		
		(ii) the Hodge--Tate spectral sequence
		$$E^{ij}_2:=H^j(X,\Omega^{\log,i}_{X})(-j)\Rightarrow H^{i+j}_{\ket}(X,\Q_p)\otimes_{\Q_p}C$$ degenerates at $E_2$.
		
		Moreover, $H^i_{B^+_{\mathrm{dR}}}(X/B_{\operatorname{dR}}^+)$ is a finite free $B_{\operatorname{dR}}^+$-module.
	\end{thm}
	
	The key step in the proof is to demonstrate the existence of a splitting $$\bigoplus_{i\geq 0}\Omega_X^{i,\log}(-i)[-i] \xrightarrow{\simeq} R\nu_*\widehat{\mathcal O}_{X_{\proket}}.$$ For rigid analytic varieties, this is shown in \cite[Proposition 7.2.5]{guo2019hodgetate} by choosing a smooth $B_{\operatorname{dR}}^+/t^2$-lifting $\mathbb X$ of $X$. Such a lifting always exists when $X$ is proper. Additionally, the splitting depends on the choice of lifting functorially. 
	
	We now turn our attention to the logarithmic case, discussing conditions under which such a splitting exists.
	
	\subsubsection{$\mathbb{B}_{\operatorname{dR}}^+/t^2$-liftings of $X$} As demonstrated in \cite{bms1} or \cite{guo2019hodgetate}, a key step in the proof of the theorem is to establish the existence of a $\mathbb{B}_{\operatorname{dR}}^+/t^2$-lifting for proper $X$. We will show that the same holds for log smooth and proper adic space.
	
	\begin{prop}
		Let $X$ be a log smooth rigid analytic variety over $C$. Then $X$ admits a log smooth $B_{\operatorname{dR}}^+/t^2$-lift $\mathbb X$ via the map $B_{\operatorname{dR}}^+/t^2 \to C$ in the following cases:
		
		(1) $X$ descends to $X_0$ over a discretely valued field $K$;
		
		(2) $X$ is proper.
	\end{prop}
	
	\begin{proof}
		(1) is clear as the base change of $X_0$ along $K \to B^+_{\operatorname{dR}}/t^2$ gives the desired lift.
		
		For (2), denote by $U:=X_{\operatorname{tr}}$ the trivial locus of $X$ (i.e. largest open where the induced log structure is trivial), and $D:=X-U$ the closed complement of $U$. Since $X$ is log smooth, according to \cite[Corollary III.1.9.5]{ogus2018log}, we can identify the log structure of $X$ with the compactifying structure induced by $U$. By Raynaud's theory on formal and rigid geometry (for example, see \cite[Lemme 4.2.2]{abbes2010elements} for the existence of closed immersion between formal schemes), and Temkin's results on analytic spaces in \cite[Corollary 4.4]{temkin2000localprop}, there exists a closed immersion $\mathfrak D \hookrightarrow \mathfrak X$ of proper flat formal schemes over $\Spf(\mathcal O_C)$ such that the closed immersion $D \hookrightarrow X$ is the rigid generic fiber of the morphism $i:\mathfrak D \hookrightarrow \mathfrak X$. Denote by $i_s:\mathfrak D_s \to \mathfrak X_{s}$ the induced map on special fibers of $i$.
		
		Choose an inclusion of fields ${F}\to C$. We will show that there exists a smooth rigid space $\mathcal S$ over ${F}$ such that the closed immersion $D\to X$ can be lifted to a morphism between flat ${F}$-rigid spaces $D_A \to X_A$ over $\mathcal S$ with a map $\Spa(C,\mathcal O_C)\to \mathcal S$. Since ${B}_{\operatorname{dR}}^+/t^2 \to C$ has a natural ${F}$-structure, by formal smoothness the map $\Spa(C,\mathcal O_C)\to \mathcal S$ can be lifted to a map $\Spa({B}_{\operatorname{dR}}^+/\xi^2,A_{\operatorname{inf}}/\xi^2)\to \mathcal S$. This gives a ${B}_{\operatorname{dR}}^+/t^2$-lifting $\mathbb D \to \mathbb X$ of $D \to X.$ The open complement $\mathbb U$ of $\mathbb D$ then induces the compactifying log structure on $\mathbb X$, which also gives a ${B}_{\operatorname{dR}}^+/t^2$-lifting of the log adic space $X$ by \cite[Proposition III.1.6.2]{ogus2018log}. Moreover, $\mathbb X$ is log smooth over ${B}_{\operatorname{dR}}^+/t^2$ by \cite[Proposition IV.4.2.4]{ogus2018log}. This will complete the proof.

		Denote by $\mathscr C_{\mathcal O_{{F}}}$ the category of complete local Artin $\mathcal O_{{F}}$-algebras with residue field $\overline{k}$. Consider the deformation functor $$\mathbf{D}:\mathscr C_{\mathcal O_{{F}}}\to \mathbf{Set}$$ defined by sending $A \in\mathscr C_{\mathcal O_{{F}}}$ to the isomorphism class of flat liftings of $i_s$ on $A.$ According to \cite[0E3S]{stacks-project} (after \cite[0E3W]{stacks-project}), the deformation functor $\mathbf{D}$ has a versal deformation, i.e., a complete Noetherian local $\mathcal O_{{F}}$-algebra $A$ with residue field $\overline{k}$, a morphism between proper flat formal schemes $ i_A:\mathfrak D_A \to \mathfrak{X}_A$ over $A$ where $A$ is topologized by powers of its maximal ideal, deforming $i_s$ such that the induced classifying map $$h_A:\mathrm{Hom}_{\mathcal O_{{F}}}(A,-)\to \mathbf{D}$$ is formally smooth. According to the proof of \cite[Proposition 13.15]{bms1}, we can extend $h_A$ and $\mathbf{D}$ to the ind-completion of $\mathscr C_{\mathcal O_{{F}}}$. This category includes $\mathcal O_C/\omega^k$ for all $k\geq 1$, where $\omega$ is a pseudouniformizer of $\mathcal O_C$. Denote by $i_k:\mathfrak D_k \to \mathfrak{X}_k$ the reduction of $i:\mathfrak D \to \mathfrak{X}$ modulo $\omega^k.$ Since we have a canonical map $\eta_0:A \to \overline{k}$, and an isomorphism $\psi_0:\eta_0^*i_A \simeq i_{s}$, applying the formal smoothness of $h_A$ (on the ind-completion of $\mathscr C_{\mathcal O_{{F}}}$) to the $\mathcal O_C/\omega \to \overline{k},$ we can choose a map $\eta_1:A\to\mathcal O_C/\omega$ lifting $\eta_0$ and an isomorphism $\psi_1:\eta_1^*i_A \simeq i_1$ of morphisms between $\mathcal O_C/\omega$-schemes lifting $\psi_0$. We can then construct $\psi_n$ and $\eta_n$ inductively, taking the inverse limit, and we can show that there exists a map $\eta:A\to\mathcal O_C $ of $\mathcal O_{{F}}$-algebra, and an isomorphism $\eta^*i_A \simeq i$ as morphism between formal $\mathcal O_C$-schemes.
		
		Finally, after inverting $p$ we have a morphism of flat ${F}$-rigid spaces $D_A \to X_A$ over $\mathcal S:=\Sp(A[1/p]).$ We can shrink $\mathcal S$ to a suitable small locally closed subset, such that $\mathcal S$ is smooth over ${F}$. This completes the proof.
	\end{proof} 
	
	
	\subsubsection{Splitting of the (log) Hodge--Tate map} We will prove the following proposition.
	
	\begin{prop} \label{htsplitting}
		Let $X$ be a log smooth rigid analytic variety over $C$. Then a choice of log smooth $B_{\operatorname{dR}}^+/t^2$-lift $\mathbb X$ of $X$ gives a natural splitting of the (log) Hodge-Tate map $$\mathrm{HT}:H^1_{\proket}(X,\widehat{\mathcal O}_{X_{\proket}})\to H^0(X,\Omega_X^{\log}(-1)).$$ Here the map $\mathrm{HT}$ is induced by the spectral sequence $$E_2^{ij}:=H^j_{\et}(X,R^i\nu_*\widehat{\mathcal O}_{X_{\proket}})\Rightarrow H^{i+j}_{\proket}(X,\widehat{\mathcal O}_{X_{\proket}}).$$
	\end{prop}
	
	We prove this proposition following \cite[Proposition 2.15]{heuer2024padicsimpsoncorrespondencesmooth}. Via the homeomorphism $|\mathbb X|=|X|,$ we may regard $\mathcal O_{\mathbb X}$ as the sheaf on $X_{\et}$. Define 
	$$L_{\mathbb X}:=\left\{\text{\begin{tabular}{l} {\parbox{4.5cm}{homomorphisms $\widetilde{\varphi}$ of sheaves of $B_{\operatorname{dR}}^+/t^2$-algebras on $X_{\proket}$ such that the right square is commutative:}}\end{tabular}}
	\begin{tikzcd}[row sep =0.55cm]
		\nu^{-1}\mathcal O_{\mathbb X} \arrow[d] \arrow[r,"\widetilde\varphi"] & \mathbb B_{\mathrm{dR}}^+/t^2 \arrow[d] \\
		\nu^{-1}\mathcal O_X \arrow[r]             & \widehat{\mathcal O}_{X_{\proket}}
	\end{tikzcd}\right\}.$$
	Note that $L_{\mathbb X}$ is a sheaf on $X_{\proket}$ as a subsheaf of $\mathcal{H}om(\lambda^{-1}\mathcal O_{\mathbb X},\mathbb B_{\mathrm{dR}}^+/t^2).$ We have 
	
	\begin{lem}
		$L_{\mathbb X}$ is a pro-Kummer \'etale torsor under $\nu^*\Omega_X^{\log}(-1)^\vee.$
	\end{lem}
	
	\begin{proof}
		Similar to the proof of \cite[Lemma 2.13]{heuer2024padicsimpsoncorrespondencesmooth} by using the theory of infinitesimal liftings for log smooth morphism, see \cite[Proposition 3.14]{kato1989log}.
	\end{proof}
	
	\begin{proof}[Proof of Propostion \ref{htsplitting}]
		For $\theta \in H^0(X,\Omega_X^{\log}(-1)),$ define a map $$s_{\mathbb X}:H^0(X,\Omega_X^{\log}(-1)) \to H^1_{\proket}(X,\widehat{\mathcal O}_{X_{\proket}})$$ by $$s_{\mathbb X}:\theta \mapsto L_{\mathbb X,\theta}:=L_{\mathbb X}\times^{\nu^*\Omega_X^{\log}(-1)^\vee}\nu^*\mathcal O_X$$ via identifying $\Omega_X^{\log}(-1)^\vee \to \mathcal O_X$, which is induced from the dual of the canonical map $\mathcal O_X \to \Omega_X^{\log}(-1)$. We need to check that $\mathrm{HT}(s_{\mathbb X}(\theta))=\theta.$
		
		The problem is local, so we may assume $X=\Spa(R,R^+)$ is an affinoid log rigid space over $C$ with a strictly \'etale morphism $X \to \mathbb E=\Spa(C\langle P \rangle,\mathcal O_C\langle P \rangle)$, which can be written as a composite of rational embeddings and finite \'etale maps. For any $a \in P,$ let $T^a$ be the induced coordinate on $X$. Let $\widetilde{X}=\Spa(R_{\infty},R_{\infty}^+) \to X$ be the Galois pro-finite Kummer \'etale cover with Galois group $\Gamma$ as before. Choose a $\Z$-basis $\{a_1,...,a_d\}$ of $P^{\operatorname{gp}}$, where $d=\dim X$. Then $\delta(a_1),...,\delta(a_d)$ form a basis of $\Omega_R^{\log},$ where $\delta: P \to \Omega_R^{\log}$ is the derivation. Denote by $\partial_1,...,\partial_d$ its dual basis in $\Omega_R^{\log,\vee}$.
		
		For any lift $\mathbb X=\Spa({R}',{R}'^+)$, by formal log smoothness there exists a lift of the map $C\langle P \rangle \to R$ induced by $X \to \mathbb E$ to an \'etale morphism $$B_{\operatorname{{dR}}}^+/t^2\langle P \rangle \to {R'}$$ of $B_{\operatorname{{dR}}}^+/t^2$-algebras. Here we view $B_{\operatorname{{dR}}}^+/t^2$ as a Tate $\Q_p$-algebra with a ring of definition $A_{\operatorname{inf}}/\xi^2,$ with log structure associated to the pre-log structure $P \to B_{\operatorname{{dR}}}^+/t^2\langle P \rangle$, $a \mapsto T^a$. Any choice of such a lift induces a section of $L_{\mathbb X}(\widetilde{X})$ as follows: denote by $[T^{a/p^{\infty}}]=(T^a,T^{a/p},...)\in \mathbb B_{\operatorname{{dR}}}^+/t^2(\widetilde{X}),$ the morphism $$B_{\operatorname{{dR}}}^+/t^2\langle P\rangle \to \mathbb B_{\operatorname{{dR}}}^+/t^2(\widetilde{X}):T^a\to [T^{a/p^{\infty}}]\text{ for any } a \in P$$ extends by formal log \'etaleness to a unique map $R' \to \mathbb B_{\operatorname{{dR}}}^+/t^2(\widetilde{X})$ lifting the map $R\to R_{\infty}$. The following lemma describes the $\Gamma$-action on $L_{\mathbb X}(\widetilde{X})$.
		
		\begin{lem}
			Write $a_i=a_i^+-a_i^-$ with $a_i^+,a_i^-\in P$ for $i=1,2,...,d$. Let $c_i^+,c_i^-:\Gamma \to \Z_p(1)=\varprojlim_{n}\boldsymbol{\mu}_{p^n}$ be the maps such that for any $\gamma \in \Gamma,$ we have $$\gamma[T^{a_i^+/p^{\infty}}]=[c_i^+(\gamma)][T^{a_i^+/p^{\infty}}],\gamma[T^{a_i^-/p^{\infty}}]=[c_i^-(\gamma)][T^{a_i^-/p^{\infty}}].$$ Then under the identification $L_{\mathbb X}(\widetilde{X})=\Omega_R^{\log}(-1)^\vee\otimes_RR_{\infty}$ induced by $R' \to \mathbb B_{\operatorname{{dR}}}^+/t^2(\widetilde{X})$, the $\Gamma$-action on the right can be described by the continuous 1-cocycle $$\Gamma\to \Omega_R^{\log,\vee}(1),\gamma \mapsto \sum_{i=1}^d \dfrac{c_i^+(\gamma)}{c_i^-(\gamma)}\partial_i,$$ via the natural identification $H^1_{\proket}(X,\nu^*\Omega_X^{\log}(-1)^\vee) \simeq H^1(\Gamma,\Omega_R^{\log,\vee}(1))$.
		\end{lem}
		
		\begin{proof}
			The proof is similar to \cite[Lemma 2.16]{heuer2024padicsimpsoncorrespondencesmooth}, let $R_{\infty}\{1\}:=\ker(\mathbb B^+_{\operatorname{dR}}/t^2(R_{\infty})\to R_{\infty})$. For any $\Sigma x_i\partial_i \in \operatorname{Hom}(\Omega_R^{\log},R_{\infty}\{1\})$, denote by $x_i^+:=x_i\partial_i(\delta(a_i^+)),x_i^-:=x_i\partial_i(\delta(a_i^-))$ for all $i$, then the corresponding element in $L_{\mathbb X}(\widetilde{X})\simeq\Omega_R^{\log}(-1)^\vee\otimes_RR_{\infty}$ is given by the morphism induced by sending $T^{a_i^+} \mapsto [T^{a_i^+/p^{\infty}}]+x_i^+$ and $T^{a_i^-} \mapsto [T^{a_i^-/p^{\infty}}]+x_i^-.$ Then to describe the $\gamma$-action, we compute:
			\begin{align*}
				\gamma([T^{a_i^+/p^{\infty}}]+x_i^+)-[T^{a_i^+/p^{\infty}}]=([c_i^+(\gamma)]-1)[T^{a_i^+/p^{\infty}}]+\gamma x_i^+,\\
				\gamma([T^{a_i^-/p^{\infty}}]+x_i^-)-[T^{a_i^-/p^{\infty}}]=([c_i^-(\gamma)]-1)[T^{a_i^-/p^{\infty}}]+\gamma x_i^-.
			\end{align*}
			This completes the proof, since we can identify $[\epsilon]-1$ by $t$ in $B_{\operatorname{dR}}/t^2$, for $\epsilon=(1,\cdots) \in \Z_p(1)$, which induces a canonical isomorphism $R\{1\}\to R(1)$.			
		\end{proof}
		
		For any $i=1,...,d,$ consider the $\widehat{\mathcal O}_{X_{\proket}}(1)$ torsor $L_{\mathbb X}\times^{\nu^*\Omega_X^{\log,\vee}(1)}\nu^*\mathcal O_X(1)$ induced by $\delta(a_i):\Omega_X^{\log,\vee}\to \mathcal O_X.$ We need to check its image under $\mathrm{HT}(1)$ is $\delta(a_i)$. This follows from the commutative diagram in Lemma \ref{constructpushforward}: $\delta(a_i)$ maps to $\mathrm{HT}(1)^{-1}\delta(a_i)$ through the map over the lower left corner, and it maps to the class defined by the 1-cocycle $\Gamma\to R(1):\gamma\to {c_i^+(\gamma)}/{c_i^-(\gamma)}$ through the map over the upper right corner, which is exactly the cocycle associated to $L_{\mathbb X}\times^{\nu^*\Omega_X^{\log,\vee}(1)}\nu^*\mathcal O_X(1)$ by the lemma above. This completes the proof.
	\end{proof}
	
	\subsubsection{Splitting of $R\nu_*\widehat{\mathcal O}_{X_{\proket}}$}
	
	The splitting of the Hodge--Tate map induces a natural decomposition of $R\nu_*\widehat{\mathcal O}_{X_{\proket}}$ as follows.
	
	\begin{prop} \label{decomposition}
		Let $X$ be a log smooth rigid analytic variety over $C$. Then a choice of log smooth $B_{\operatorname{dR}}^+/t^2$-lift $\mathbb X$ of $X$ gives a natural splitting of $R\nu_*\widehat{\mathcal O}_{X_{\proket}}$ as $\bigoplus_{i\geq 0}\Omega_X^{i,\log}(-i)[-i]$ in the derived category.
	\end{prop}
	
	\begin{proof}
		The splitting of the (log) Hodge--Tate map constructed in Proposition \ref{htsplitting} gives a natural quasi-isomorphism $$\mathcal O_X \oplus \Omega_X^{1,\log}(-1)[-1] \xrightarrow{\simeq} \tau^{\leq 1}R\nu_*\widehat{\mathcal O}_{X_{\proket}}.$$
		
		The quotient map $$\left(\Omega_X^{1,\log}\right)^{\otimes_{\mathcal O_X} i}\to\bigwedge^i \Omega_X^{1,\log}=\Omega_X^{i,\log}$$ has a canonical $\mathcal O_X$-linear section $s_i$ given by $$\omega_1\wedge...\wedge\omega_i \mapsto \dfrac{1}{i!}\sum_{\sigma\in S_i}\mathrm{sgn}(\sigma)\omega_{\sigma(1)}\otimes...\otimes\omega_{\sigma(i)}.$$
		Here, we denote by $(-)^{\otimes i}$ the tensor product of $i$ copies of the object. Therefore, we can construct a map $$\Omega_X^{i,\log}(-i)[-i]\to (\Omega_X^{1,\log}(-1)[-1])^{\otimes_{\mathcal O_X} i}\to (R\lambda_*\widehat{\mathcal O}_{X_{\proket}})^{\otimes_{\mathcal O_X} i} \to R\lambda_*\widehat{\mathcal O}_{X_{\proket}},$$ which induces a natural map $$\bigoplus_{i\geq 0}\Omega_X^{i,\log}(-i)[-i] \to R\nu_*\widehat{\mathcal O}_{X_{\proket}}.$$
		
		The map constructed above is a quasi-isomorphism by (the construction of) Proposition \ref{pushforward}.
	\end{proof}

	\subsubsection{Proof of Theorem \ref{degeneration}}
	
	With all these preparations, the proof for Theorem \ref{degeneration} then follows by counting dimensions.
	
	We start with (2). The spectral sequence $$E_2^{ij}:=H^j_{\et}(X,R^i\nu_*\widehat{\mathcal O}_{X_{\proket}}) \Rightarrow H^{i+j}(X_{\proket},\widehat{\mathcal O}_{X_{\proket}})\simeq H^{i+j}_{\ket}(X,\Q_p)\otimes_{\Q_p}C$$ converges, where the last isomorphism follows from the primitive comparison theorem in \cite[Theorem 6.2.1]{dllz2023logadic}. The claim then follows from Proposition \ref{decomposition} by counting dimensions.
	
	To show (1), we need to show that $$\sum_{i=1}^d\dim_CH^j(X,\Omega_X^{i,\log})=\dim_{C}H^{i+j}_{\mathrm{logdR}}(X).$$
	By (2), we have $$\sum_{i=1}^d\dim_CH^j(X,\Omega_X^{i,\log})=\dim_{\Q_p}H^{i+j}_{\ket}(X,\Q_p).$$ On the other hand, Theorem \ref{bdrcomparison} shows 
	\begin{align*}
		\dim_{\Q_p}H^{i+j}_{\ket}(X,\Q_p)&=\dim_{B_{\operatorname{dR}}}H^{i+j}_{B_{\operatorname{dR}}^+}(X)(*)\left[\dfrac{1}{t}\right]  \\
		&= \dim_{B_{\operatorname{dR}}^+}H^{i+j}_{B_{\operatorname{dR}}^+}(X)(*)/{torsion}\\
		&\leq  \dim_{B_{\operatorname{dR}}^+}H^{i+j}_{B_{\operatorname{dR}}^+}(X)(*)/t \\
		&\leq\dim_{C}H^{i+j}_{\mathrm{logdR}}(X),
	\end{align*} 
	where the second equality holds since $H^i_{B_{\operatorname{dR}}^+}(X)(*)/{torsion}$ is a finite rank torsion-free module over $B_{\operatorname{dR}}^+$ without nontrivial divisible submodule, so it must be free by \cite[Theorem 3]{rotman1960dvr}. The last inequality follows from Theorem \ref{bdrC}, which induces a short exact sequence $$0\to H^{i+j}_{B_{\operatorname{dR}}^+}(X)(*)/t\to H^{i+j}_{\mathrm{logdR}}(X) \to  H^{i+j+1}_{B_{\operatorname{dR}}^+}(X)(*)[t]\to 0.$$
	
	This proves (1) as all inequalities must be equality.
	
	Finally, the above computation also shows that $H^{i}_{B_{\operatorname{dR}}^+}(X)$ is a finite free $B_{\operatorname{dR}}^+$-module as it is torsion-free. This completes the proof.

	\bibliography{ref}
	
	\bibliographystyle{halpha}	
	
\end{document}